\documentclass[11pt, letterpaper]{amsart}
\usepackage{amsfonts,amsthm,mathtools}
\usepackage[all,cmtip]{xy}
\usepackage{fullpage}
\usepackage{amsmath}
\usepackage{amssymb}
\usepackage{enumerate}
\usepackage{tikz}
\usepackage[backref]{hyperref}
\usepackage{cleveref}
\usepackage{verbatim}
\usepackage{cancel}
\usepackage{float}
\newtheorem{question}{Question}

\newtheorem{lem}{Lemma}[section]
\newtheorem{thm}[lem]{Theorem}
\newtheorem{proposition}[lem]{Proposition}
\newtheorem{prop}[lem]{Proposition}
\newtheorem{cor}[lem]{Corollary}

\theoremstyle{definition}
\newtheorem{remark}[lem]{Remark}

\newtheorem{conj}{Conjecture}

\DeclareMathAlphabet{\curly}{U}{rsfs}{m}{n}

\newcommand{\tors}{\operatorname{tors}}
\newcommand{\End}{\operatorname{End}}
\newcommand{\Aut}{\operatorname{Aut}}
\newcommand{\Gal}{\operatorname{Gal}}

\newcommand{\gon}{\operatorname{gon}}
\newcommand{\ord}{\operatorname{ord}}

\newcommand{\Q}{\mathbb{Q}}
\newcommand{\C}{\mathbb{C}}
\newcommand{\Z}{\mathbb{Z}}
\newcommand{\F}{\mathbb{F}}
\newcommand{\GL}{\operatorname{GL}}
\newcommand{\AGL}{\operatorname{AGL}}
\newcommand{\PSL}{\operatorname{PSL}}
\newcommand{\SL}{\operatorname{SL}}

\newcommand{\Supp}{\operatorname{Supp}}

\newcommand{\PP}{{\mathbb P}}

\newcommand{\lmfdbec}[3]{\href{https://www.lmfdb.org/EllipticCurve/Q/#1/#2/#3}{#1.#2#3}}

  \newcommand{\im}{\operatorname{im}}

\mathchardef\mhyphen="2D

\title{Sporadic points of odd degree on $X_1(N)$ coming from $\Q$-curves}

\author{Abbey Bourdon}
\address{Wake Forest University, Winston-Salem, NC 27109, USA}
\email{bourdoam@wfu.edu}
\urladdr{http://users.wfu.edu/bourdoam/}

\author{Filip Najman}
\address{University of Zagreb, Bijeni\v{c}ka Cesta 30, 10000 Zagreb, Croatia}
\email{fnajman@math.hr}
\urladdr{http://web.math.pmf.unizg.hr/~fnajman/}

\begin{document}

\begin{abstract}
We say a closed point $x$ on a curve $C$ is sporadic if there are only finitely many points on $C$ of degree at most $\deg(x)$. In the case where $C$ is the modular curve $X_1(N)$, most known examples of sporadic points come from elliptic curves with complex multiplication (CM). We seek to understand all sporadic points on $X_1(N)$ corresponding to $\Q$-curves, which are elliptic curves isogenous to their Galois conjugates. This class contains not only all CM elliptic curves, but also any elliptic curve $\overline{\Q}$-isogenous to one with a rational $j$-invariant, among others. In this paper, we show that all non-CM $\Q$-curves giving rise to a sporadic point of odd degree lie in the $\overline{\Q}$-isogeny class of the elliptic curve with $j$-invariant $-140625/8$. In addition, we show that a stronger version of this finiteness result would imply Serre's Uniformity Conjecture.
\end{abstract}

\maketitle

\section{Introduction}

Let $E$ be an elliptic curve defined over a number field $F$. By the Mordell-Weil Theorem, the set of points on $E$ with coordinates in $F$ is a finitely generated abelian group. That is, there exists a finite abelian group $E(F)_{\tors}$ and nonnegative integer $r$ such that
\[
E(F) \cong E(F)_{\tors} \times \Z^r.
\]
A standard classification problem is to describe the groups that arise as $E(F)_{\tors}$ as $E$ ranges over all number fields $F$ of a fixed degree $d$. In general, there are only finitely many possible torsion subgroups that appear---a consequence of Merel's Uniform Boundedness Theorem \cite{merel}---and the list of groups is known explicitly for degrees $1 \leq d \leq 3$; see \cite{mazur77,kamienny86,KM88,kamienny92,Deg3Class}. For example, the possible torsion subgroups of elliptic curves over $\Q$ are given by the following theorem:

\begin{thm}[Mazur, \cite{mazur77}]
If $E/\Q$ is an elliptic curve, then $E(\Q)_{\tors}$ is isomorphic to one of the following:
\begin{align*}
&\Z/m\Z,  \text{ with } 1 \leq m \leq 10 \text{ or } m=12,\\
& \Z/2\Z \times \Z/2m\Z,  \text{ with } 1 \leq m \leq 4.
\end{align*}
Moreover, for each group $T$ that arises, there are infinitely many non-isomorphic elliptic curves $E/\Q$ such that $E(\Q)_{\tors} \cong T$.
\end{thm}

The classification of torsion subgroups of elliptic curves over cubic fields presents a new phenomenon. For the first time, there are certain groups which arise for only finitely many isomorphism classes of elliptic curves.

\begin{thm}[Najman \cite{najman16}, Derickx, Etropolski, van Hoeij, Morrow, Zureick-Brown \cite{Deg3Class}] \label{NajmanCurve}
The elliptic curve {\lmfdbec{162}{c}{3}} over $\Q(\zeta_9)^+$ with $j$-invariant $-140625/8$ is the unique elliptic curve over a cubic field with a point of order 21.
\end{thm}

Phrased another way, this means the elliptic curve {\lmfdbec{162}{c}{3}} corresponds to a sporadic point of degree 3 on the modular curve $X_1(21)$. Recall $X_1(N)$ is an algebraic curve over $\Q$ whose non-cuspidal points parametrize isomorphism classes of pairs of an elliptic curve together with a distinguished point of order $N$. We say a point $x$ on a curve defined over a number field is \textbf{sporadic} if there are only finitely many points of degree at most $\deg(x)$. By definition, such points cannot belong to an infinite parameterized family of points of the same degree, making them more difficult to detect using standard tools. Our lack of understanding of sporadic points on the modular curves $X_1(N)$ has made it difficult to extend the classification of torsion subgroups beyond cubic fields.

Inspired by this, we seek to characterize the elliptic curves capable of producing sporadic points on $X_1(N)$, in search of some sort of framework which might explain their existence. Moreover, this problem is inextricably linked to Serre's Uniformity Conjecture on Galois representations of elliptic curves. We say $j$ is a \textbf{sporadic $j$-invariant} if there exists a sporadic point $x\in X_1(N)$ such that $j=j(x)$. So, for example, by Theorem \ref{NajmanCurve}, we see that $-140625/8$ is a sporadic $j$-invariant, as is any $j$-invariant associated to an elliptic curve with complex multiplication (CM) by \cite[Theorem 7.1]{BELOV}. Work of the first author in collaboration with Ejder, Liu, Odumodu, Viray \cite{BELOV} shows that there are only finitely many sporadic\footnote{In fact, \cite{BELOV} shows there are only finitely many \emph{isolated} $j$-invariants in $\Q$. The set of all isolated points strictly contains the set of all sporadic points, so the claim about sporadic $j$-invariants is immediate.} $j$-invariants in $\Q$, assuming the following conjecture; such a result is known unconditionally for points of odd degree \cite{OddDeg}.
\begin{conj}[Uniformity Conjecture] There exists a constant such that for all non-CM elliptic curves $E/\Q$, the mod $p$ Galois representation of $E$ is surjective for all $p>C$.

\end{conj}

\noindent Conjecture 1 originated with a question of Serre \cite{serre72}, but has since been formally conjectured by both Zywina \cite{ZywinaImages} and Sutherland \cite{sutherland}. A wealth of computational evidence and partial theoretical progress which supports an affirmative answer. See for example \cite{BP11,BPR13,BalakrishnanEtAl,ZywinaImages,sutherland,lemosTrans,lemosZ}.

While Serre's Uniformity Conjecture implies that there are only finitely many sporadic $j$-invariants in $\Q$, showing that the same holds true for all sporadic $j$-invariants associated to non-CM $\Q$-curves would actually imply Serre's Uniformity Conjecture. Recall a $\Q$-curve is any elliptic curve isogenous (over $\overline{\Q}$) to its Galois conjugates. This is a property enjoyed by CM elliptic curves and elliptic curves with rational $j$-invariant, though the collection of all $\Q$-curves properly contains these sets.

\begin{thm} \label{Thm:SerreConnection}
Suppose there exist only finitely many isogeny classes containing non-CM $\Q$-curves (up to isomorphism over $\overline{\Q}$) that give rise to a sporadic point on $X_1(p^2)$ for some prime $p$. Then Serre's Uniformity Conjecture holds.
\end{thm}

\noindent This motivates the study torsion points within isogeny classes of $\Q$-curves, which we initiate in the present work. In particular, we pose the following question, which extends those raised in \cite{BELOV}.

\begin{question}
Do all non-CM $\Q$-curves giving rise to sporadic points on any modular curve $X_1(N)$ belong to only finitely many $\overline{\Q}$-isogney classes, even as we allow $N$ to range over all positive integers?
\end{question}

In our main result, we answer Question 1 for sporadic points of odd degree, showing that the elliptic curve  {\lmfdbec{162}{c}{3}} actually plays a much larger role in explaining sporadic points on the modular curves $X_1(N)$ than what is apparent from Theorem \ref{NajmanCurve}. Not only does this elliptic curve give a non-cuspidal sporadic point of least possible degree, but it is also isogenous to any other non-CM $\Q$-curve corresponding to a sporadic point of odd degree on $X_1(N)$ for any positive integer $N$. Since any sporadic point on $X_1(N)$ must have degree less than its $\Q$-gonality,\footnote{Recall the $k$-gonality of a curve $C$ over a number field $k$ is the least degree of a non-constant rational map $f: C \rightarrow \mathbb{P}^1_k$. Hilbert's irreducibility theorem \cite[Ch.9]{serre97} implies $f^{-1}(\mathbb{P}^1(k))$ contains infinitely many degree $d$ points.} this is a consequence of the following more general result.

\begin{thm} \label{thm:FiniteIsogenyClasses}
Let $x \in X_1(N)$ be a point of odd degree corresponding to a non-CM $\Q$-curve $E$. If $\deg(x) < \gon_{\Q}(X_1(N))$, then $E$ is $\overline{\Q}$-isogenous to an elliptic curve with the $j$-invariant $-140625/8$. In particular, any non-CM $\Q$-curve giving rise to a sporadic point of odd degree on $X_1(N)$ is isogenous (over $\overline{\Q}$) to the elliptic curve {\lmfdbec{162}{c}{3}}.
\end{thm}

The proof builds on recent work of Cremona and the second author \cite{CremonaNajmanQCurve} (which in turn builds on work of Elkies \cite{elkies}). A crucial result at the foundation of our approach is that any non-CM $\Q$-curve defined over a number field of odd degree is isogenous to an elliptic curve with rational $j$-invariant \cite[Theorem 2.7]{CremonaNajmanQCurve}, making it possible to use information about the Galois representations of elliptic curves over $\Q$ to deduce results for the original $\Q$-curve. For example, in Theorem 1.1 of \cite{CremonaNajmanQCurve}, this isogeny connection is used to show that there exists a point of odd degree on $X_1(N)$ corresponding to a non-CM $\Q$-curve only if $\Supp(N) \subseteq \{2,3,5,7,11,13,17,37\}$. Paired with results of \cite{BELOV}, we immediately see that there will be only finitely many $\Q$-curves with $j$-invariant in an extension of \emph{bounded} degree giving rise to a sporadic point of odd degree on any modular curve of the form $X_1(N)$.\footnote{Indeed, there exists a uniform bound on the level of the $m$-adic Galois representation associated to all non-CM elliptic curves over number fields of fixed degree \cite[Prop. 6.1]{BELOV}. Since any curve by definition can have only finitely many sporadic points, the claim follows from Proposition 5.8 and Theorem 4.3 of \cite{BELOV}.} Thus the strength of Theorem \ref{thm:FiniteIsogenyClasses} is in showing that these sporadic $j$-invariants lie in finitely many isogeny classes, even if we remove the bound on the degree of $\Q(j)$.

In the case where $N$ is a power of a single prime, bounds on the $\Q$-gonality of $X_1(N)$ due to Derickx and van Hoeij \cite{derickxVH} are strong enough to imply that there are no sporadic points of odd degree corresponding to non-CM $\Q$-curves. This is in direct contrast to the CM case.

\begin{thm} \label{Thm:main}
Let $p$ be a prime number and $k$ a positive integer. If $x=[E,P] \in X_1(p^k)$ is a point of odd degree corresponding to a $\Q$-curve with $\deg(x) < \gon_{\Q}(X_1(p^k))$, then $E$ has complex multiplication. Moreover, for any prime $p \equiv 3 \pmod{4}$ and $k$ sufficiently large, there exist sporadic CM points of odd degree on $X_1(p^k)$.
\end{thm}

\noindent Our proof shows that there are infinitely many CM $j$-invariants producing sporadic points on $X_1(p^k)$ of odd degree, and that they necessarily belong to infinitely many distinct $\overline{\Q}$-isogeny classes; see Remark \ref{remark4.2}. This shows that the ``non-CM" assumption in Theorem \ref{thm:FiniteIsogenyClasses} is necessary if one wishes to characterize a set of $\Q$-curves producing sporadic points of odd degree which belong to only finitely many isogeny classes.

Even though we have established that all non-CM $\Q$-curves corresponding to a sporadic point on $X_1(N)$ of odd degree lie in a single $\overline{\Q}$-isogeny class, it remains to determine whether there can be infinitely sporadic $j$-invariants within that isogeny class. Our final question is the following:

\begin{question}
Does there exist a non-CM $\overline{\Q}$-isogeny class containing infinitely many sporadic $j$-invariants?
\end{question}

\noindent In Proposition \ref{Prop:8.1}, we show that the answer is yes if there exists a sporadic point of sufficiently low degree associated to a curve in the $\overline{\Q}$-isogeny class. However, as the only known examples of points satisfying this degree condition correspond to CM elliptic curves, it is not clear whether one should expect the existence of such a point.

\subsection{Related Work.} While sporadic points on more general modular curves are certainly of interest as they pertain to the classification of Galois representations of elliptic curves (see, for example \cite{BalakrishnanEtAl}) or the modularity of elliptic curves (see, for example, \cite{FLS15,DNS20,Box21}), for this summary we will restrict our attention to modular curves  $X_1(N)$ and $X_0(N)$ (which parametrizes elliptic curves with a cyclic $N$-isogeny). The first examples of sporadic points on modular curves of this form arise as points in $X_0(N)(\Q)$ in cases where this curve has genus greater than 0. Key contributions to this classification were made by Mazur \cite{mazur78} and Kenku \cite{Kenku79,Kenku80,Kenku80_2,Kenku81}, but see Table 4 in \cite{LRAnn} for a more complete list of references. The classification of quadratic and cubic points on modular curves $X_0(N)$ is the topic of several recent works, and many new examples of sporadic points have been discovered; see, for example \cite{OzmanSiksek19,Box_quad21}.
For modular curves $X_1(N)$, the least possible degree of a non-cuspidal sporadic point is 3, and, as mentioned above, there is a unique elliptic curve producing such a point which was first discovered by the second author \cite{najman16}. Others can be found in tables associated to work of Derickx and van Hoeij \cite{derickxVH}; note that any point in degree less than the $\Q$-gonality of the curve will be sporadic, provided the associated Jacobian has rank 0 over $\Q$. Sporadic points associated to elliptic curves with rational $j$-invariant were the topic of \cite{BELOV} and \cite{OddDeg}, and those corresponding to elliptic curves with supersingular reduction were studied in \cite{Smith2018}.

For a complete summary of what is known concerning sporadic CM points on modular curves, see work of Clark, Genao, Pollack, and Saia \cite{LeastCMDeg}.

\subsection{Outline.} We summarize relevant background information on Galois representations and isogenies of elliptic curves in $\S2$, along with some basic information concerning modular curves. Section 3 contains the proof of Theorem \ref{Thm:SerreConnection}. Our approach to Theorem \ref{thm:FiniteIsogenyClasses} relies on first establishing certain divisibility conditions for the existence of a rational torsion point on a non-CM $\Q$-curve defined over a number field of odd degree (see Proposition \ref{Prop:div}) which are obtained by studying torsion points within isogeny classes of elliptic curves. Theorem \ref{Thm:main} is a relatively quick consequence of these divisibility conditions and is proved in $\S5$. To complete the proof of Theorem \ref{thm:FiniteIsogenyClasses}, we first show that there exists a point of odd degree on $X_1(N)$ associated to a non-CM $\Q$-curve only for $N$ of a certain form, as in Proposition \ref{prop:combine2}. In Section 7, we show that any odd degree point $x \in X_1(N)$ with $\deg(x)<\gon_{\Q}(X_1(N))$ associated to a non-CM $\Q$-curve would force a low degree point on one of finitely many modular curves; see sections $7.1$ or $7.2$. In cases where these low degree points would arise on modular curves of the form $X_1(N)$, their existence can generally be ruled out by prior work. However, other cases require an analysis of certain modular curves which parametrize elliptic curves with unusual entanglement between their $2^a$- and $3^k$-torsion point fields. Explicit computations of the rational points on these entanglement modular curves appear in section 7.3, completing the proof of Theorem \ref{thm:FiniteIsogenyClasses}. Partial progress towards Question 2 appears in  $\S8$.

{We note that all code for Magma computations relevant to sections 7.2 and 7.3 is publicly available at
{\href{https://web.math.pmf.unizg.hr/~fnajman/code_QC.zip}{\url{https://web.math.pmf.unizg.hr/~fnajman/code_QC.zip}}}.}

\section*{Acknowledgements}
We thank Josha Box, Pete L. Clark, Davide Lombardo, Jackson Morrow, and Jeremy Rouse for helpful conversations. We also thank Pete L. Clark, Jackson Morrow, and Bianca Viray for helpful comments on an earlier draft.

The first author was partially supported by an A. J. Sterge Faculty Fellowship. {The second author was supported by the QuantiXLie Centre of Excellence, a
  project co-financed by the Croatian Government and European Union
  through the European Regional Development Fund - the Competitiveness
  and Cohesion Operational Programme (Grant KK.01.1.1.01.0004) and by
  the Croatian Science Foundation under the project
  no. IP-2018-01-1313.}

\section{Background and Notation}

\subsection{Conventions}
For a number field $k$ and algebraic closure $\overline{k}$, we let $\Gal_k \coloneqq \Gal(\overline{k}/k)$ denote the absolute Galois group of $F$. If $n \in \Z^+$, we use $\Supp(n)$ to denote the set of all prime numbers dividing $n$. If $n=p_1^{a_1} p_2^{a_2} \cdots p_r^{a_r}$ for distinct primes $p_i$, then $\ord_{p_i}(n)\coloneqq a_i$.

If $C$ is a curve defined over a number field $k$ and $x \in C$ is a closed point, then $k(x)$ denotes the residue field of $x$. By the degree of $x$ we mean the degree of $k(x)$ over $k$.

If $E$ is an elliptic curve defined over a number field $k$ and $P \in E(\overline{k})$, then $k(P)$ denotes the field extension of $k$ generated by the $x$- and $y$-coordinates of $P$. By $m$-isogeny of elliptic curves, we mean a cyclic isogeny of degree $m$. Specific elliptic curves are referred to by their LMFDB label.

For subgroups of $\GL_2(\Z_p)$, we generally use the notation of Sutherland and Zywina in \cite{SutherlandZywina}, though in a few cases for the prime $p=2$ we use that of Rouse and Zureick-Brown \cite{RouseDZB} for ease of reference.

\subsection{Galois Representations of Elliptic Curves.} For an elliptic curve $E$ defined over a number field $k$, the absolute Galois group $\Gal_k$ acts naturally on the set of all torsion points of $E$, denoted $E(\overline{k})_{\tors}$. This action is encoded in the adelic Galois representation associated to $E$:
\[
\rho_E: \Gal_k \rightarrow \Aut (E(\overline{k})_{\tors}) \cong \GL_2(\hat{\Z}).
\]
Through the isomorphism $\GL_2(\hat{\Z}) \cong \prod_{p \text{ prime}} \GL_2(\Z_p)$ and natural projection, we obtain the $m$-adic representation associated to $E$ for any integer $m$, denoted
\[
\rho_{E,m^{\infty}}: \Gal_k \rightarrow \prod_{p \text{ prime},  p \mid m} \GL_2(\Z_p).
\]
This representation describes the action of $\Gal_k$ on all points of $E$ of order $n$ with $\Supp(n) \subseteq \Supp(m)$. Via reduction mod $n$, we recover the mod $n$ Galois representation associated to $E$
\[
\rho_{E,n}: \Gal_k \rightarrow \GL_2(\Z/n\Z),
\]
which gives the action of $\Gal_k$ on $E[n]$, the set of points on $E$ of order dividing $n$. If $m$ and $n$ are relatively prime, we may combining reduction mod $n$ in some components with projection in others to define
\[
\rho_{E,n \cdot m^{\infty}}: \Gal_k \rightarrow \GL_2(\Z/n\Z) \times \GL_2(\Z_m).
\]
Here, $\Z_m \coloneqq \prod_{p \in \Supp(m)} \Z_p$.

If $E$ does not have complex multiplication, then Serre's Open Image Theorem \cite{serre72} states that the image of  $\rho_{E}$, denoted $\im \rho_E$, is open (and thus of finite index) in $\GL_2(\hat{\Z})$. In particular, the image of the $p$-adic Galois representation associated to a non-CM elliptic curve $E$ is of finite index in $\GL_2(\Z_p)$. The following proposition summarizes a few of the known constraints on this index in the case of elliptic curves over $\Q$.

\begin{prop} \label{Prop:index}
For any non-CM elliptic curve $E/\Q$,
\[
\ord_2([\GL_2(\Z_{2}): \im \rho_{E/\Q, 2^{\infty}}]) \leq 6.
\]
If $E/\Q$ is a non-CM elliptic curve with a rational cyclic $p$-isogeny for some odd prime $p$, then we have the following:
\begin{enumerate}
\item If $p \geq 7$, then $\ord_p([\GL_2(\Z_{p}): \im \rho_{E/\Q, p^{\infty}}])=0$.
\item If $p = 5$, then $\ord_5([\GL_2(\Z_{5}): \im \rho_{E/\Q, 5^{\infty}}]) =0$ unless $E/\Q$ has either a rational cyclic 25-isogeny or two independent 5-isogenies. In the latter cases, we have $\ord_5([\GL_2(\Z_{5}): \im \rho_{E/\Q, 5^{\infty}}])= 1$.
\item If $p = 3$, then $\ord_3([\GL_2(\Z_{3}): \im \rho_{E/\Q, 3^{\infty}}]) \leq 2$.
\end{enumerate}
\end{prop}

\begin{proof}
For $p=5$ or $p > 7$, this follows from work of Greenberg (see Theorems 1 and 2, along with Remark 4.2.1 in \cite{greenberg2012}). The case of $p=7$ is addressed in work of Greenberg, Rubin, Silverberg, Stoll \cite{greenberg2014}. For $p=3$, the result follows from Corollary 1.10 in \cite{RSZ21}. The case of $p=2$ is given by Corollary 1.3 in work of Rouse and Zureick-Brown \cite{RouseDZB}.
\end{proof}

A further implication of Serre's Open Image Theorem is the existence of a positive integer $N$ such that $\im \rho_{E} = \pi^{-1}(\im \rho_{E,N})$, where $\pi$ denotes the natural reduction map. The smallest such $N$ is called the level of the adelic Galois representation. In a similar way,  for any positive integer $m$, we defined the level of the $m$-adic representation to be the least positive integer $N$ such that $\im \rho_{E,m^{\infty}} = \pi^{-1}(\im \rho_{E,N})$.

\subsection{Isogenies of Elliptic Curves}

For an elliptic curve $E$ defined over a number field $k$, we say the subgroup $C$ of $E$ is $k$-rational if $\sigma(C)=C$ for all $\sigma \in \Gal_k$. That is, for all $P\in C$, we have $\sigma(P) \in C$ (though $\sigma$ need not fix $P$ itself). There is a one-to-one correspondence between finite subgroups $C$ of $E$ and isomorphism classes of isogenies $\varphi:E \rightarrow E'$ given by associating $C$ with the unique isogeny having $\ker(\varphi)=C$, and $\varphi$ is defined over $k$ if and only if $C$ is $k$-rational. See $\S \text{III}.4$ of \cite{silverman} for details. For any two non-CM elliptic curves $E$, $E'$ isogenous over $\overline{\Q}$, there exists a cyclic isogeny $\varphi: E \rightarrow E'$ which is unique up to sign \cite[Lemma A.1]{CremonaNajmanQCurve}, so we may generally assume our isogenies are cyclic.

In the case of elliptic curves over $\Q$, we have a complete classification of the possible rational cyclic subgroups that can occur.

\begin{thm}[Mazur \cite{mazur78}, Kenku \cite{Kenku79,Kenku80,Kenku80_2,Kenku81,kenku}, and others; see Section 9 of \cite{LRAnn}] \label{IsogClassification}
If $E/\Q$ is an elliptic curve possessing a $\Q$-rational cyclic subgroup of order $N$, then $N \leq 19$ or $N \in \{21, 25, 27, 37, 43, 67, 163\}$.
\end{thm}

This classification can also be used to characterize certain rational isogenies for non-CM elliptic curves $E/\Q$ under base extension to number fields of odd degree. If a non-CM elliptic curve $E/k$ has a $k$-rational cyclic isogeny of degree $n$, then so will any $E'/k$ with $j(E')=j(E)$. Thus the following result can be used to deduce isogeny information for any non-CM elliptic curve with $j$-invariant in $\Q$.

\begin{prop}[Cremona, Najman \cite{CremonaNajmanQCurve}] \label{prop:CN}
Let $E/\Q$ be an elliptic curve without CM, and let $p$ be an odd prime such that $E$ has no $p$-isogenies defined over $\Q$. Then all $p$-isogenies of $E$ are defined over number fields of even degree, unless $j(E)=2268945/128$, in which case $E$ acquires $7$-isogenies over the cubic field generated by a root of $x^3-5x-5$.
\end{prop}

Finally, we note that any $\Q$-curve defined over a number field $k$ of odd degree is isogenous over $k$ to an elliptic curve with rational $j$-invariant, as a consequence of the following key result.

\begin{thm}[Cremona, Najman \cite{CremonaNajmanQCurve}] \label{thm:CN}
Let $E$ be a non-CM $\Q$-curve defined over a number field $k$. If $\Q(j(E))$ has no quadratic subfields, then $E$ is isogenous over $k$ to an elliptic curve with rational $j$-invariant.
\end{thm}

\subsection{Modular Curves.}Points on the affine curve $Y_1(N)$ correspond to equivalence classes of elliptic curves $E/\C$ with a distinguished point $P$ of order $N$. We consider two pairs $[E,P]$ and $[E',P']$ to be equivalent if there exists an isomorphism $\varphi: E \rightarrow E'$ such that $\varphi(P)=P'$. By adding a finite number of points to $Y_1(N)$---called cusps---we obtain the smooth projective curve $X_1(N)$. In fact, $X_1(N)$ can be given the structure of an algebraic curve over $\Q$, and we view it as such unless otherwise noted. See \cite[Section 7.7]{modular} or \cite{DR} for details.

If $x=[E,P] \in X_1(N)$ is a non-cuspidal point, we say a Weierstrass equation over a number field $k$ is a model for $E/k$ if it has $j$-invariant equal to $j(E)$, and we use $P$ to denote the image of the point associated to $x$ under the implied isomorphism. For any model of $E/k$ having $P \in E(k)$, the residue field $\Q(x)$ is contained in $k$. Conversely, there exists a model of $E/\Q(x)$ such that $P \in E(\Q(x))$. See \cite[p. 274, Proposition VI.3.2]{DR}. It is often useful to have a more concrete way of describing the residue field, as given in the following lemma.

\begin{lem}\label{lem:ResidueField}
Let $E$ be an elliptic curve and let $P \in E$ be a point of order $N$. Then the residue field of the closed point $x=[E,P] \in X_1(N)$ is given by
\[
\Q(x)=\Q(j(E),\mathfrak{h}(P)),
\]
where $\mathfrak{h}: E \rightarrow E/\Aut(E) \cong \mathbb{P}^1$ is a Weber function for $E$.
\end{lem}

\begin{proof}
Since $X_1(N)$ parametrizes elliptic curves up to isomorphism, we may choose a model of $E$ defined over $F \coloneqq \Q(j(E))$. It is clear that $F(\mathfrak{h}(P))) \subseteq F(P)$, and this field does not depend on the choice of model for $E/F$: for any other $E'/F$ with $j(E')=j(E)$ and isomorphism $\psi:E \rightarrow E'$, we have $F(\mathfrak{h}(P))=F(\mathfrak{h}(\psi(P)))$ by \cite[p. 107]{shimura}. Thus $F(\mathfrak{h}(P)) \subseteq \Q(x)$.

To show equality holds, it suffices to construct a model of $E$ over $F(\mathfrak{h}(P)))$ such that $P$ becomes rational. Indeed, for any $\sigma \in \Gal_{F(\mathfrak{h}(P)) }$,
\[
\sigma(P)=\zeta P
\]
for some $\zeta \in \Aut(E)=\mu_2, \mu_4$ or $\mu_6$. This defines a character $\chi$, and $P$ is rational on the twist $E^{\chi^{-1}}$ over $F(\mathfrak{h}(P))$.
\end{proof}

\begin{remark}
If $E$ is given by an equation of the form $y^2=4x^3-c_2x-c_3$, then for $P=(x,y) \in E$ and $\Delta=c_2^3-27c_3^2$ we may take $\mathfrak{h}$ to be
\[ \mathfrak{h}(P) = \begin{cases} \frac{c_2c_3}{\Delta}x & j(E) \neq 0,1728, \\ \frac{c_2^2}{\Delta} x^2 & j(E) = 1728,\\ \frac{c_3}{\Delta} x^3 & j(E) = 0. \end{cases} \]
If $\psi: E \rightarrow E'$ is an isomorphism, then one can check that $\mathfrak{h}_E=\mathfrak{h}_{E'} \circ \psi$; see \cite[p. 107]{shimura}. Thus if $E$ does not have CM, the residue field of $[E,P] \in X_1(N)$ is generated over $\Q(j(E))$ by the $x$-coordinate of $P$.
\end{remark}

Lemma \ref{lem:ResidueField} also gives a concrete way of computing the degree of a point on $X_1(N)$ from a particular model of $E/\Q(j(E))$. The following lemma allows us to relate the degree of this point to that of its image on other modular curves.

            \begin{prop}\label{prop:Degree}
                For positive integers $a$ and $b$, there is a natural $\Q$-rational map $f:X_1(ab) \rightarrow X_1(a)$ defined by sending $[E,P]$ to $[E,bP]$. Moreover
                \[
                    \deg(f)=
                    c_{f}\cdot b^2 \prod_{p \mid b,\, p \nmid a}
                    \left(1-\frac{1}{p^2}\right),
                \]
                where $c_{f}=1/2$ if $a \leq 2$ and $ab>2$ and $c_{f}=1$ otherwise.
            \end{prop}

\begin{proof}
By the moduli interpretation, the map is $\Q$-rational,  and the degree calculation follows from \cite[p.66]{modular}.
\end{proof}

In a similar way, the modular curve $X_0(N)$ is a smooth projective curve over $\Q$ whose non-cuspidal points parametrize pairs $[E,C]$ where $E$ is an elliptic curve and $C \subset E$ is a cyclic subgroup of order $N$. As noted above, the group $C$ defines a unique isogeny $\varphi:E \rightarrow E'$ with $\ker(\varphi)=C$, so we may alternatively notate $[E,C]$ with the triple $[E, E', \varphi]$. The following lemma gives a useful description of the residue field of a non-cuspidal point on $X_0(N)$.

\begin{lem}\label{lem:IsogenyResidueField}
Let $\varphi:E \rightarrow E'$ be a cyclic isogeny of non-CM elliptic curves of degree $n$. Then
\[
\Q(\varphi)=\Q(j(E),j(E'))
\]
is the residue field of the induced point $[E, \ker(\varphi)] \in X_0(n)$.
\end{lem}

\begin{proof} This follows directly from \cite[Corollary A.5]{CremonaNajmanQCurve}. See also \cite[Proposition 3.3]{ClarkVolcanoes}.
\end{proof}

\section{Proof of Theorem \ref{Thm:SerreConnection}}
In this section, we connect the existence of sporadic points $X_1(p^2)$ to Serre's Uniformity Conjecture, as discussed in the introduction. Let $X(p)$ be the compactification of the modular curve which parametrizes isomorphism classes of pairs $(E,(P,Q))$, where $E$ is an elliptic curve and $(P,Q)$ is and $\F_p$-basis of $E[p]$. This is a smooth projective curve defined over $\Q$ whose base change to $\Q(\zeta_p)$ has $p-1$ connected components.
\begin{thm}
Suppose {either that there exists} finitely many primes $p$ such that $X(p)$ has a sporadic point corresponding to a non-CM elliptic curve with $j$-invariant in $\Q$ or that all non-CM $\Q$-curves corresponding to sporadic points on $X_1(p^2)$ lie in finitely many isogeny classes, as $p$ varies through all primes. Then there exists a bound $C$ such that for every non-CM elliptic curve $E/\Q$, the mod $p$ representation $ \rho_{E,p}$ attached to $E$ is surjective for all $p>C$.
\end{thm}
\begin{proof}
We will prove this theorem by showing that for any $p$ large enough, a non-CM elliptic curve $E/\Q$ with non-surjective $\rho_{E,p}$ will induce both a sporadic point on $X(p)$, corresponding to $E$, and a sporadic point corresponding to a non-CM $\Q$-curve, isogenous to $E$, on $X_1(p^2)$.Since only finitely many rational $j$-invariants can lie in a single $\overline{\Q}$-isogeny class,\footnote{This follows from Lemma \ref{lem:IsogenyResidueField} and the fact that there are only finitely many elliptic curves over $\Q$ in a $\Q$-rational isogeny class. The latter is originally due to Shafarevich; see, for example, Corollary IX.6.2 in \cite{silverman}.} these non-CM $\Q$-curves would necessarily lie in infinitely many geometric isogeny classes. Recall the result of Abramovich \cite[Theorem 0.1]{abramovich},  which says that the $\mathbb C$-gonality $d_\C(X_\Gamma)$ of a modular curve $X_\Gamma$ corresponding to a congruence subgroup $\Gamma$ satisfies
$$\frac{7}{800}D_\Gamma \leq d_\C(X_\Gamma),$$
where $D_\Gamma$ is the index of $\Gamma$ in $\PSL_2(\Z)$. We have that $D_{\Gamma(p)}=\frac{p(p^2-1)}{2}$ and $D_{\Gamma_1(p^2)}=\frac{p^2(p^2-1)}{2}$. Recall that any point of degree $<d_\C({X_\Gamma})/2$ is sporadic on $X_\Gamma$ {(see \cite[Proposition 2]{frey})}.

Suppose now that $E/\Q$ is an elliptic curve without CM, $p>37$ a prime and that $\rho_{E,p}$ is not surjective. By known results about images of Galois (see e.g. \cite[Theorem 1.11]{ZywinaImages})  we know that $\im \rho_{E, p}$ is contained in $C_{ns}^+(p)$, the normalizer of the non-split Cartan subgroup. Since $\#C_{ns}^+(p)= 2(p^2-1)$, it follows that $[\Q(E[p]):\Q]|2(p^2-1)$, and hence $E$ (together with a basis for $E[p]$) induces a $\Q(E[p])$-rational point on $X(p)$. By Abramovich's bound (see also Remark 7.4. in \cite{LeastCMDeg}), it follows that this point is necessarily sporadic for sufficiently large $p$.

Furthermore, as $E$ has two independent $\Q(E[p])$-rational $p$-isogenies, it is $\Q(E[p])$-isogenous to an elliptic curve $E'$ which has a $\Q(E[p])$-rational $p^2$-isogeny $f$ and a point of order $p$ lying in the kernel of $f$ defined over $\Q(E[p])$. In particular, we have
$$\im \rho_{E'/\Q(E[p]), p^2}\subseteq \left\{\begin{pmatrix}
                                1+pt & * \\
                                0 & 1+pk
                              \end{pmatrix}\right\}.$$
Taking the fixed field $K'$ of the  subgroup
                      $$\left\{\begin{pmatrix}
                                1 & * \\
                                0 & 1+pk
                              \end{pmatrix}\right\},$$
 we obtain that $K'$ is of degree $\leq 2p(p^2-1)$ and that $E'$ (which is a $\Q$-curve) has a $K'$-rational point of order $p^2$. By Abramovich's bound, we again get that this point is sporadic for sufficiently large $p$.
\end{proof}

\begin{remark}
There are infinitely many CM $\Q$-curves (up to isomorphism over $\overline{\Q}$) that give rise to a sporadic point on $X_1(p^2)$ for some prime $p$, and they lie in infinitely many distinct isogeny classes. To see this, let $E/\overline{\Q}$ be an elliptic curve with complex multiplication by the maximal order in an imaginary quadratic field $K$. For simplicity, we may assume $j(E) \neq 0, 1728$. Then for any prime $p \geq 5$ that is split in $K$ and $x \in X_1(p^2)$, we have by Theorem 6.2 in \cite{BC2} that
\[
\deg(x)=h_{K}\cdot \varphi(p^2),
\]
where $h_K$ denotes the class number of $K$. For any prime sufficiently large,
\[
h_K < \frac{7}{3200} p(p+1),
\]
and so
\[
\deg(x) < \frac{7}{3200} p(p+1)\varphi(p^2)= \frac{7}{1600}[\PSL_2(\Z) : \Gamma _1(p^2)]  < \frac{1}{2}\gon_{\mathbb{Q}}(X_1(p^2)),
\]
where the last inequality follows from by Theorem 0.1 of \cite{abramovich}. Thus $x$ is sporadic by \cite[Prop. 2]{frey}. Moreover, since the endomorphism algebra $K=\End(E) \otimes \Q$ is an isogeny invariant, it follows that the CM $j$-invariants corresponding to sporadic points on the curves $X_1(p^2)$ lie in infinitely many distinct $\overline{\Q}$-isogeny classes.

\end{remark}

\section{Divisibility Conditions}
In this section, we prove that there are certain divisibility conditions which must be satisfied in order for a $\Q$-curve to possess a point of prime-power order defined over a number field of odd degree. As we have seen, any $\Q$-curve $E$ defined over a number field $F$ of odd degree is isogenous to an elliptic curve $E'$ with rational $j$-invariant by \Cref{thm:CN}. After establishing some preliminary results concerning the Galois representations of isogenous elliptic curves ($\S4.1$), we prove that the existence of a point of order $p^k$ in $E(F)$ implies $[F:\Q]$ is divisible by a power of $p$ controlled by the index of the image of the $p$-adic Galois representation of an elliptic curve over $\Q$ with $j$-invariant $j(E')$ (see Lemma \ref{lem:PrelimDiv}). In $\S4.3$, these results are taken together to prove the following:

\begin{prop} \label{Prop:div}
Let $E$ be a non-CM $\Q$-curve defined over a number field $F$ of odd degree. Then $E$ is $F$-isogenous to an elliptic curve $E'$ with $j(E') \in \Q$. If $E(F)$ has a point of order $p^k$ for some prime number $p$ and $k \in \Z^+$, then $p \in \{2,3,5,7,11,13\}$ and the following divisibility conditions hold:

\begin{enumerate}
\item If $p =13$, then $3 \cdot 13^{2k-2} \mid [F:\Q]$.
\item If $p =11$, then $5 \cdot 11^{2k-2} \mid [F:\Q]$.
\item If $p=7$ and $j(E')\neq 3^3\cdot5\cdot7^5/2^7$, then $7^{2k-2} \mid [F:\Q]$.
\item If $p=7$ and $j(E')= 3^3\cdot5\cdot7^5/2^7$, then $3 \cdot 7^{\max(0,2k-3)} \mid [F:\Q]$.
\item If $p=5$, then $5^{\max(0,2k-3)} \mid [F:\Q]$.
\item If $p=3$, then $3^{\max(0,2k-4)} \mid [F:\Q]$.
\item If $p=2$, then $k \leq 4$.
\end{enumerate}
\end{prop}

\subsection{Torsion in Isogeny Classes}
In this section, we study the implications for rational torsion points on an elliptic curve for the Galois representation of an isogenous curve. Though Lemma \ref{cor:4.4} and its corollary are similar to results already in the literature (see, for example, \cite[Proposition 1.4]{ReverterVila} and \cite[Lemma 2.5]{Daniels-Morrow}), the assumptions are slightly different so we include complete proofs. Note that Proposition \ref{prop:isog_rep} is not used in the proof of our main result, but we include it here as this more refined result may be of independent interest.

\begin{lem}\label{cor:4.4}
Let $E_1/F$ be an elliptic curve with a point $P\in E_1(F)$ of prime order $p$ and let $\varphi:E_1\rightarrow E_2$ be an isogeny defined over $F$. Then either $E_2(F)$ has a point of order $p$ or for some basis of $E_2[p]$
$$\rho_{E_2,p}(\sigma)=\begin{pmatrix}
                                \chi_p(\sigma)& y \\
                                0 & 1
                              \end{pmatrix},   \text{ for all } \sigma \in \Gal_F.$$
\end{lem}

\begin{proof}
Note we may assume $\varphi$ is cyclic. With respect to the basis $\{P,Q\}$ for $E_1[p]$,
$$\rho_{E_1,p}(\sigma)=\begin{pmatrix}
                                1 & b \\
                                0 & \chi_p(\sigma)
                              \end{pmatrix} \text{ for all } \sigma \in \Gal_F.$$
 If $\varphi(P)$ has order $p$, we are done, so suppose not. Then $\varphi(Q)$ must have order $p$, and we have
\[
\sigma(\varphi(Q))=\varphi(\sigma(Q))=\varphi(bP+\chi_p(\sigma) Q)=\chi_p(\sigma) \varphi(Q).
\]
Thus if $E_2[p]=\{\varphi(Q),R\}$, we see
\[
\rho_{E_2,p}(\sigma)=\begin{pmatrix}
                                \chi_p(\sigma) & b \\
                                0 & 1
                              \end{pmatrix} \text{ for all } \sigma \in \Gal_F. \qedhere
\]
\end{proof}

\begin{cor}\label{cor:IsogenyCor}
Let $E$ be an elliptic curve defined over a number field $F$ which is isogenous over $F$ to an elliptic curve $E'$. If $P \in E(F)$ is a point of prime order $p$, then:
\begin{enumerate}
\item $E'$ has a point of order $p$ over $F(\zeta_p)$.
\item $E'$ has a rational point of order $p$ over an extension of $F$ of degree dividing $p$.
\end{enumerate}
\end{cor}

\begin{proof}
Let $\varphi: E \rightarrow E'$ be an isogeny, where $E$, $E'$, and $\varphi$ are defined over $F$. By Lemma A.1 in \cite{CremonaNajmanQCurve}, we may assume $\varphi$ is cyclic. If $E'(F)$ has a point of order $p$, we are done, so suppose not. Then in particular, $\varphi(P)=O$. By \Cref{cor:4.4} there is a basis of $E'[p]$ such that
\[
\im \rho_{E', p} = \left\{\begin{pmatrix}
\chi_p(\Gal_F) & *  \\
0 & 1
\end{pmatrix}\right\},
\] where $\chi_p$ is the mod $p$ cyclotomic character and $*$ is nonzero. Part (1) follows immediately. As shown in the proof of \cite[Proposition 4.1]{CremonaNajmanQCurve}, $E'$ attains a rational point of order $p$ over an extension of degree $p$, proving part 2.
\end{proof}

\begin{prop}\label{prop:isog_rep}
Suppose $E_1/F$ is an elliptic curve with a point $P\in E_1(F)$ of order $N$ and $\{P,Q\}$ is a basis for $E_1[N]$, and let $\varphi:E_1\rightarrow E_2$ be an $m$-isogeny with $m\mid N$ and $\ker \varphi= \langle (N/m) P\rangle $.
Then with respect to the basis $\{\varphi(Q),R\}$, where $R\in E[N]$ satisfies $mR=\varphi(P)$,
$$\rho_{E_2,N}(\sigma)=\begin{pmatrix}
                                \chi_N(\sigma)& y \\
                                z & 1
                              \end{pmatrix}, \text{ where } y\equiv0 \pmod{N/m}, z\equiv 0\pmod{m}  .$$
\end{prop}

\begin{proof}
Let $\sigma\in \Gal_F$. Let $P$ be a point of order $N$ in $E_1(F)$, and fix a basis $\{P,Q\}$ for $E_1[N]$ and $\{\varphi(Q), R\}$ for $E_2[N]$. As $\varphi(P)$ is a point of order $N/m$ independent with $m\varphi(Q)$ in $E_2[N/m]$ we can choose $R$ such that $mR=\varphi(P)$ and which will necessarily be independent to $\varphi(Q)$. With respect to these bases, we have
$$\rho_{E_1,N}(\sigma)=\begin{pmatrix}
                                1 & b \\
                                0 & d
                              \end{pmatrix} \text{
and }\rho_{E_2,N}(\sigma)=\begin{pmatrix}
                                x & y \\
                                z & w
                              \end{pmatrix}.$$
Since
$$(mR)^\sigma=\varphi(P)^\sigma=\varphi(P^\sigma)=\varphi(P)=m R,$$
we can conclude that $y\equiv 0 \pmod{N/m}$.

We have
$$ \varphi(Q)^\sigma=\varphi(Q^\sigma)=\varphi(bP+dQ)=b\varphi(P)+d\varphi(Q)=bmR+d\varphi(Q),$$
so we conclude that $x=d$ and $z \equiv 0 \pmod{m}$. Since
$$d=\det \rho_{E_1,N}(\sigma) =\chi_N(\sigma)=\det \rho_{E_2,N}(\sigma)=xw-yz=dw-yz,$$
and $yz \equiv 0 \pmod{N}$, we conclude that $w=1$.
\end{proof}

\subsection{General Divisibility Conditions}
\begin{lem} \label{Lem:index}
Let $p$ be a prime number, and let $E/F$ be a non-CM elliptic curve isogenous over $F$ to an elliptic curve $E'/\Q$.
Then
\[
[\GL_2(\Z_p):\im \rho_{E/F,p^{\infty}}]= [\GL_2(\Z_{p}): \im \rho_{E'/\Q, p^{\infty}}] \cdot[F \cap \Q(E'[p^{\infty}]):\Q].
\]
\end{lem}

\begin{proof}
Over $F$, the images of $\rho_{E,p^{\infty}}$ and $\rho_{E',p^{\infty}}$ have the same index in $\GL_2(\Z_{p})$ since the curves are isogenous, i.e.,
\[
[\GL_2(\Z_p):\im \rho_{E/F,p^{\infty}}]=[\GL_2(\Z_p):\im \rho_{E'/F,p^{\infty}}].
\]
See, for example, $\S2.1$ of \cite{greenberg2012}.
The claim will follow from the fact that
\[
[\GL_2(\Z_p):\im \rho_{E'/F,p^{\infty}}]=[\GL_2(\Z_p): \im \rho_{E'/\Q, p^{\infty}}] \cdot[F \cap \Q(E'[p^{\infty}]):\Q].
\]
Indeed, by Galois theory (see, for example, Proposition 7.14 in \cite{MilneFieldGaloisTheory}),
\[\im \rho_{E'/F, p^{\infty}} \cong \Gal(F(E'[p^{\infty}])/F) \cong \Gal(\Q(E'[p^{\infty}])/F \cap \Q(E'[p^{\infty}])).
\]
Thus we may view $\im \rho_{E'/F,p^{\infty}}$ as a subgroup of $\im \rho_{E'/\Q, p^{\infty}}$ of index $[F \cap \Q(E'[p^{\infty}]):\Q]$.
\end{proof}

\begin{lem}\label{lem:PrelimDiv}
Let $E/F$ be a non-CM elliptic curve with a point $P \in E(F)$ of order $p^k$, where $p$ is a prime number. If $E$ is $F$-isogenous to an elliptic curve $E'$ with $j(E') \in \Q$, then $[F:\Q]$ is divisible by
\[\begin{cases}
p^{\max(0,2k-2-d)} \text{ if $p$ is odd}\\
p^{\max(0,2k-2-d-1)} \text{ if $p=2$},
\end{cases}\]
where $d=\ord_p([\GL_2(\Z_{p}): \im \rho_{E''/\Q, p^{\infty}}])$ for any elliptic curve $E''/\Q$ with $j(E'')=j(E')$.
\end{lem}

\begin{proof}
Let $E/F$ be a non-CM elliptic curve with $P \in E(F)$ of order $p^k$, and suppose $\varphi:E \rightarrow E'$ is an $F$-rational isogeny. Replacing $F$ with at worst a quadratic extension $L/F$, we may view $\varphi$ as an $L$-isogeny from $E$ to an elliptic curve $E''/\Q$ with $j(E'')=j(E')$. Since $E$ has a rational point of order $p^k$ over $L$, we have $\Gal(L(E[p^k])/L)$ is contained in the group of matrices in $\GL_2(\Z/p^k/ \Z)$ of the form
\[
\left\{\begin{pmatrix}
1 & *  \\
0 & *
\end{pmatrix}\right\}.
\]
Thus $\# \Gal(L(E[p^k])/L) \mid p^k \cdot \varphi(p^k)=p^{2k-1}(p-1)$, and it follows that
\[
\ord_{p}([\GL_2(\Z/p^k\Z):\im \rho_{E/L, p^k}] ) \geq 2k-2.\]
Thus $p^{2k-2} \mid [\GL_2(\Z_{p}): \im \rho_{E''/\Q, p^{\infty}}] \cdot[L \cap \Q(E''[p^{\infty}]):\Q]$ by Lemma \ref{Lem:index}. If $d=\ord_p([\GL_2(\Z_{p}): \im \rho_{E''/\Q, p^{\infty}}])$, then since $L$ is at most a quadratic extension of $F$, it follows that
\[
p^{\text{max}(0,2k-2-d)} \mid [F:\Q]
\]
if $p$ is odd and
\[
p^{\text{max}(0,2k-2-d-1)} \mid [F:\Q]
\]
if $p=2$.
\end{proof}

\subsection{Proof of Proposition \ref{Prop:div}}
Suppose $E/F$ is a non-CM $\Q$-curve, where $[F:\Q]$ is odd, and suppose $E(F)$ contains a point of order $p^k$. Then $p \in \{2,3,5,7,11,13,17,37\}$ by Theorem 1.1 in \cite{CremonaNajmanQCurve}, and by \Cref{thm:CN}, there exists a rational isogeny $\varphi: E \rightarrow E'$ such that $j(E') \in \Q$. By Lemma \ref{lem:PrelimDiv}, we have that $[F:\Q]$ is divisible by
\[\begin{cases}
p^{\max(0,2k-2-d)} \text{ if $p$ is odd}\\
p^{\max(0,2k-2-d-1)} \text{ if $p=2$},
\end{cases}\]
where $d=\ord_p([\GL_2(\Z_{p}): \im \rho_{E''/\Q, p^{\infty}}])$ for any elliptic curve $E''/\Q$ with $j(E'')=j(E')$. Since $E''$ and $E'$ are isomorphic over an extension $L/F$ of degree at most 2, we may also view $\varphi:E \rightarrow E''$ as an isogeny over $L$.

Since $E$ has a $p$-isogeny defined over $F$, so does $E'$ by Proposition 3.2 of \cite{CremonaNajmanQCurve}. If $p$ is odd, then since $F$ has odd degree, it follows by Proposition 3.3 in \cite{CremonaNajmanQCurve} that $E''$ has a rational cyclic $p$-isogeny over $\Q$ or else $p=7$ and $j(E')=3^3\cdot5\cdot7^5/2^7$. Thus unless $j(E')=3^3\cdot5\cdot7^5/2^7$ and $p=7$, an upper bound for $d$ is given by Proposition \ref{Prop:index}. In particular, we note that $k \leq 4$ if $p=2$.

If $j(E')=3^3\cdot5\cdot7^5/2^7$, then $\# \im \rho_{E''/\Q,7}=18$ or 36. By \cite[Lemma 21]{OddDeg}, the 7-adic Galois representation of $E''/\Q$ has level $7$, and so together these imply
\[
\ord_7([\GL_2(\Z_7): \im \rho_{E''/\Q, 7^{\infty}}])=1,
\]
\[
\ord_{3}([\GL_2(\Z_7):\im \rho_{E''/\Q, 7^{\infty}}] ) =0.\]
Since $E$ has a rational point of order $7^k$ over $L$, as in the proof of Lemma \ref{lem:PrelimDiv} we have $\# \Gal(L(E[7^k])/L) \mid 7^{2k-1}\cdot6$, and so
\[
\ord_{7}([\GL_2(\Z/7^k\Z):\im \rho_{E/L, 7^k}] ) \geq 2k-2,\]
\[
\ord_{3}([\GL_2(\Z/7^k\Z):\im \rho_{E/L, 7^k}] ) \geq 1.\]
By Lemma \ref{Lem:index}, we have $3\cdot7^{2k-2} \mid [\GL_2(\Z_{7}): \im \rho_{E''/\Q, 7^{\infty}}] \cdot[L \cap \Q(E''[7^{\infty}]):\Q]$, and the computations above on the index of $\im \rho_{E''/\Q,7^{\infty}}$ in $\GL_2(\Z_7)$ imply $3\cdot 7^{\max(0,2k-3)} \mid [L \cap \Q(E''[13^{\infty}]):\Q]$. Thus
\[
3 \cdot 7^{\max(0,2k-3)} \mid [F:\Q].
\]
as desired.

Now, suppose $p \geq 11$. Since $E(L)$ has a point of order $p$, by Corollary \ref{cor:IsogenyCor} the curve $E''$ gains a point of order $p$ over an extension $L'/L$ of degree dividing $p$. As $\ord_2([L:\Q])\leq 1$, the classification of mod $p$ images of elliptic curves over $\Q$ with a rational cyclic $p$-isogeny (see, for example, Tables 1 and 2 in \cite{GJNajman}, which is complete for elliptic curves with a $p$-isogeny by \cite{zywina15}) implies
\begin{align*}
5 \mid [F:\Q] \text{ if $p=11$},\\
3 \mid [F:\Q] \text{ if $p=13$},
\end{align*}
and that we reach a contradiction if $p=17$ or $37$.

\section{Proof of Theorem \ref{Thm:main}}
In this section, we prove Theorem \ref{Thm:main}. The first part of the theorem, which implies that there are no sporadic points of odd degree on $X_1(p^k)$ corresponding to non-CM $\Q$-curves, is proven in $\S5.1$. We show that the divisibility conditions established in Section 4 are sufficient to prove the degree of such a point is greater than or equal to the $\Q$-gonality of $X_1(p^k)$, using bounds implied by work of Derickx and van Hoeij \cite{derickxVH}. On the other hand, it can be deduced from work of Bourdon and Pollack \cite{BP} that there are infinitely many CM points of odd degree on these curves; see $\S5.2$.

\subsection{Non-CM Sporadic Points of Odd Degree}

\begin{thm}
Let $p$ be a prime number. If $x=[E,P] \in X_1(p^k)$ is a point of odd degree corresponding to a non-CM $\Q$-curve, then $\deg(x) \geq \gon_{\Q}(X_1(p^k))$. In particular, $x$ is not sporadic.
\end{thm}
\begin{proof}
Let $x=[E,P] \in X_1(p^k)$ be a point of odd degree corresponding to a non-CM $\Q$-curve. By Proposition \ref{Prop:div}, we know $p \in \{2,3,5,7,11,13\}$. We will consider each case separately. Claims about the $\Q$-gonality of $X_1(N)$ follow from Table 1 in work of Derickx and van Hoeij \cite{derickxVH}.
\begin{itemize}
\item Suppose $p=13$. Then $3 \cdot 13^{2k-2} \mid [\Q(x):\Q]$ by Proposition \ref{Prop:div}. Since $\gon_{\Q}(X_1(13))=2$, it follows that $\gon_{\Q}(X_1(13^k))$ is at most
\[
2 \cdot \deg(X_1(13^k) \rightarrow X_1(13))=2 \cdot (13^{k-1})^2=2 \cdot 13^{2k-2},
\]
and $\deg(x) > \gon_{\Q}(X_1(13^k))$.
\item Suppose $p=11$. Then $5 \cdot 11^{2k-2} \mid [\Q(x):\Q]$ by Proposition \ref{Prop:div}. Since $\gon_{\Q}(X_1(11))=2$, it follows that $\gon_{\Q}(X_1(11^k))$ is at most
\[
2 \cdot \deg(X_1(11^k) \rightarrow X_1(11))=2 \cdot (11^{k-1})^2=2 \cdot 11^{2k-2},
\]
and $\deg(x) > \gon_{\Q}(X_1(11^k))$.
\item Suppose $p=7$. Since $X_1(7)$ is genus 0, we may assume $k \geq 2$. Then $7^{2k-2} \mid [\Q(x):\Q]$ or $3 \cdot 7^{2k-3} \mid [\Q(x):\Q]$ by Proposition \ref{Prop:div}. Since $\gon_{\Q}(X_1(7^2)) \leq 21$, it follows that $\gon_{\Q}(X_1(7^k))$ is at most
\[
21\cdot \deg(X_1(7^k) \rightarrow X_1(7^2))=21 \cdot (7^{k-2})^2=3\cdot 7^{2k-3},
\]
and $\deg(x) \geq \gon_{\Q}(X_1(7^k))$.

\item Suppose $p=5$. Since $X_1(5)$ is genus 0, we may assume $k \geq 2$. Then $5^{2k-3} \mid [\Q(x):\Q]$ by Proposition \ref{Prop:div}. Since $\gon_{\Q}(X_1(25))=5$, it follows that $\gon_{\Q}(X_1(5^k))$ is at most
\[
5 \cdot \deg(X_1(5^k) \rightarrow X_1(25))=5 \cdot (5^{k-2})^2=5^{2k-3},
\]
and $\deg(x) \geq \gon_{\Q}(X_1(5^k))$.
\item Suppose $p=3$. Since $X_1(3)$ and $X_1(9)$ are genus 0, we will assume $k \geq 3$. Then $3^{2k-4} \mid [\Q(x):\Q]$ by Proposition \ref{Prop:div}. Since $\gon_{\Q}(X_1(27))=6$, it follows that $\gon_{\Q}(X_1(3^k))$ is at most
\[
6 \cdot \deg(X_1(3^k) \rightarrow X_1(27))=6 \cdot (3^{k-3})^2=2 \cdot 3^{2k-5} < 3 \cdot 3^{2k-5}=3^{2k-4},
\]
and $\deg(x) > \gon_{\Q}(X_1(3^k))$.
\item Suppose $p=2$. Then $k \leq 4$ by Proposition \ref{Prop:div}. If $k \leq 3$, then $X_1(2^k)$ is genus 0. As $\gon_{\Q}(X_1(16))=2$ and no non-cuspidal rational points by \cite{Levi1908}, again the claim follows. \qedhere
\end{itemize}
\end{proof}

\subsection{Existence of Sporadic CM Points of Odd Degree}
\begin{prop}\label{prop:CMSporadicPoints}
For any $p \equiv 3 \pmod{4}$, there exist sporadic CM points of odd degree on $X_1(p^k)$ for $k$ sufficiently large.
\end{prop}

\begin{remark}
We note that by work of Aoki \cite[Cor. 9.4]{aoki95}, there is no CM point of odd degree on $X_1(p^k)$ if $p$ is a prime with $p \equiv 1 \pmod{4}$.
\end{remark}

\begin{proof}
For a prime $p \equiv 3 \pmod{4}$ and $n \in \Z^+$ we define $\delta$ as follows:
\[ \delta \coloneqq
\begin{cases}
\lfloor 3n/2 \rfloor-1, \, \, \, p>3,\\
0, \, \, \, p=3 \text{ and } n=1, \\
\lfloor 3n/2 \rfloor-2, \, \, \, p=3 \text{ and } n \geq 2.
\end{cases}
\]
By \cite[Thm 2.6]{BP} there exists a CM point $x \in X_1(p^n)$ of odd degree
\[
\deg(x)=h_{\Q(\sqrt{-p})}\cdot \frac{p-1}{2} p^{\delta},
\]
where $h_{\Q(\sqrt{-p})}$ denotes the class number of $\Q(\sqrt{-p})$. For $n$ sufficiently large, we will have
\[
h_{\Q(\sqrt{-p})} < \frac{7}{1600} p^{2n-\lfloor 3n/2 \rfloor}.
\]
In particular, this forces $n>1$ if $p=3$. Thus by Theorem 0.1 of \cite{abramovich},
\[
\deg(x) < \frac{7}{1600}[\PSL_2(\Z) : \Gamma _1(p^n)]  \leq \frac{1}{2}\gon_{\mathbb{Q}}(X_1(p^n)),
\]
and $x$ is sporadic by \cite[Prop. 2]{frey}.
\end{proof}

\begin{remark} \label{remark4.2}
Suppose $p \equiv 3 \pmod{4}$ is prime. By \cite[Cor. 9.4]{aoki95}, a CM elliptic curve with a point of order $p^n$ {over an odd degree number field} has CM by an order in $\Q(\sqrt{-p})$. Thus for distinct primes $p \equiv 3 \pmod{4}$ the sporadic points of Proposition \ref{prop:CMSporadicPoints} necessarily come from CM elliptic curves that are not isomorphic over $\overline{\Q}$. Said another way, if $\mathcal{I}_{\text{odd}}$ is the set of all sporadic points of odd degree on all curves $X_1(N)$ for $N \in \Z^+$, then $j(\mathcal{I}_{\text{odd}})$ contains infinitely many CM $j$-invariants. Moreover, since the endomorphism algebra $\End(E) \otimes \Q$ is an isogeny invariant, these CM elliptic curves lie in infinitely many distinct $\overline{\Q}$-isogeny classes.
\end{remark}

\section{Beyond Powers of a Single Prime}
In fact, studying points of odd degree on $X_1(p^k)$ associated to non-CM $\Q$-curves is not far from the case of such points of odd degree on $X_1(n)$. Since any such $\Q$-curve $E$ is isogenous to an elliptic curve $E'$ with rational $j$-invariant, it follows that $E'$ corresponds to a point of odd degree on $X_0(p)$ for any odd prime $p$ dividing $n$. As shown in \cite[Proposition 3.3]{CremonaNajmanQCurve}, this generally implies that $E'$ must in fact correspond to a point in $X_0(p)(\Q)$. Thus the support of $n$ is constrained by the possible rational isogenies of elliptic curves over $\Q$, as in Theorem \ref{IsogClassification}. This can be used to prove the following result, which generalizes \cite[Theorem 3]{OddDeg}.

\begin{prop} \label{prop:combine2}
Let $x=[E,P] \in X_1(N)$ be a point of odd degree, where $E$ is a non-CM $\Q$-curve. Then $N=2^ap^k$ for $p \in \{3,5,7,11,13\}$ unless $E$ is in the isogeny class of an elliptic curve over $\Q$ with a rational cyclic $21$-isogeny. If $k>0$, then we obtain the following bounds on $a$:
\begin{enumerate}
\item If $p=3$, then $a \leq 2$.
\item If $p=5$, then $a \leq 1$.
\item If $p=7$, then $a \leq 2$.
\item If $p=11$, then $a \leq 1$.
\item If $p=13$, then $a \leq 1$.
\end{enumerate}

\end{prop}

After establishing a preliminary lemma in $\S6.1$, we prove Proposition \ref{prop:combine2} in $\S6.2$. From the proof of Proposition \ref{prop:combine2}, we also deduce the following:

\begin{cor} \label{cor:5.3}
Let $x \in X_1(4 \cdot p^k)$ be a point of odd degree associated to a non-CM $\Q$-curve $E$, where $k>0$.
Then:
\begin{enumerate}
\item If $p=3$ and $E$ is \emph{not} isogenous to $E'$ with $j(E') \in \{3^2\cdot 23^3/2^6,-3^3\cdot 11^3/2^2\}$, then $E$ is $\overline{\Q}$-isogenous to an elliptic curve $E''/\Q$ with a rational cyclic $6$-isogeny or with a rational cyclic $3$-isogeny and full $2$-torsion (but no $4$-isogeny) in a cubic extension.
\item If $p=7$ and $E$ is \emph{not} isogenous to $E'$ with $j(E') \in \{-3^3 \cdot 13 \cdot 479^3/2^{14},3^3 \cdot 13/2^2\}$, then $E$ is $\overline{\Q}$-isogenous to an elliptic curve $E''/\Q$ with a rational cyclic $7$-isogeny and full $2$-torsion (but no $4$-isogeny) in a cubic extension.
\end{enumerate}
\end{cor}

\subsection{A Preliminary Result}
\begin{lem}\label{lem:isog}
Let $E$ and $E'$ be non-CM elliptic curves which are $F$-isogenous and suppose that $E$ has a $p^n $-isogeny over $F$. Then $E'$ has either $2$ independent $F$-isogenies of degrees $p^{n-t}$ and $p^t$ for some $t$, or $E'$ has a $p^{a}$-isogeny over $F$, where $a\geq \min \{n,1+\frac n 2\}$.
\end{lem}

\begin{proof}
Let $f:E\rightarrow E'$ be a cyclic $F$-isogeny (which is unique up to sign). If $\deg f$ is coprime to $p$, then $E'$ has a $p^n$-isogeny defined over $F$ and we are done. If $\deg f= p^t\cdot b$ where $(b,p)=1$, then there exists a curve $E_0$ which is $b$-isogenous to $E'$ and its $p$-adic representation will be the same as that of $E'$, so will have the same degrees of $p$-power isogenies as $E'$. Hence we can suppose without loss of generality that $\deg f= p^t$, as otherwise we work with $E_0$ instead of $E'$.

Let $f_1:E\rightarrow E_1$ be the $p^n$-isogeny that $E$ has by our assumptions. If $\ker f_1\cap \ker f=\{O\}$, then $f_1\circ \hat f$ is a $p^{t+n}$-isogeny over $F$ and we are done. If $\ker f_1\supseteq \ker f$ then {$f_1$} factors as $f_1=f_2\circ f$, where $f_2:E' \rightarrow E_1$ is a $p^{n-t}$-isogeny over $F$ which is independent with the $p^t$-isogeny $\hat f:E'\rightarrow E$. If $\ker f\supseteq \ker f_1$, then we have $t\geq n$ and the isogeny $p^t$-isogeny $\hat f:E'\rightarrow E$ again completes this case.

Suppose finally that $\ker f_1\cap \ker f\neq \{O\}$ and that neither of $\ker f_1$ and $\ker f$ is contained in the other (this implies $n,t\geq 2$). Then there exists an elliptic curve $E_2$ and a $p^m$-isogeny $f_2:E\rightarrow E_2$ defined over $F$ such that both $f$ and $f_1$ factor through $f_2$; we choose $E_2$ and $m$ such that $m$ is the largest integer satisfying this condition. Let $f=f_3 \circ f_2$ and $f_1=f_4\circ f_2$. Now we see that $\hat f:E'\rightarrow E$ is a $p^{t}$-isogeny over $F$ and $ f_4 \circ \hat{f_3}:E'\rightarrow E_1$ is an $p^{(t-m)+(n-m)}$-isogeny over $F$. We have $1\leq m<t$ and $m<n$, so if $m\geq \frac n 2$, then we have $t=(t-m)+m\geq 1+\frac n 2$ and if $m<\frac n 2$, then $(t-m)+(n-m) \geq 1 +\frac n 2$.
\end{proof}

\subsection{Proof of Proposition \ref{prop:combine2}}

Let $F \coloneqq \Q(x)$, and fix a model of $E/F$ with $P \in E(F)$. By Proposition \ref{Prop:div}, $\Supp(N) \subseteq \{2,3,5,7,11,13\}$. By \Cref{thm:CN}, there exists an $F$-rational isogeny $\varphi: E \rightarrow E'$ such that $j(E') \in \Q$. Let $p \mid N$ be prime. Since $E$ has a $p$-isogeny defined over $F$, so does $E'$ by Proposition 3.2 of \cite{CremonaNajmanQCurve}. If $p$ is odd, then since $F$ has odd degree, it follows by Proposition 3.3 in \cite{CremonaNajmanQCurve} that any $E''/\Q$ with $j(E'')=j(E')$ has a rational cyclic $p$-isogeny or else $p=7$ and $j(E')=3^3\cdot5\cdot7^5/2^7$.

For now, suppose $j(E')\neq 3^3\cdot5\cdot7^5/2^7$. If distinct odd primes $p,q$ divide $N$, then $E''/\Q$ has a rational cyclic $pq$-isogeny. By the classification of rational isogenies over $\Q$ (see Theorem \ref{IsogClassification}), this cannot happen unless $E''$ has a rational cyclic 15- or 21-isogeny. By the theorem statement, it suffices to consider the case where $E''$ has a rational cyclic 15-isogeny. For the sake of contradiction assume $N=2^a3^b5^c$ for $b,c>0$. By \Cref{cor:IsogenyCor}, there is an extension $F_1/F$ of degree dividing 3 such that $E'(F_1)$ has a point of order 3. Similarly, there is an extension $F_2/F$ of degree dividing 5 such that $E'(F_2)$ has a point of order 5. It follows that $E'$ has a point of order 15 over $F_1F_2$, which has odd degree over $\Q$. This contradicts Proposition 15 in \cite{OddDeg}. If $j(E')= 3^3\cdot5\cdot7^5/2^7$, then $E''$ has no rational isogenies over $\Q$, so it does not gain a $p$-isogeny, for odd primes $p\neq 7$, over any odd degree number field; the first part of the result follows.

Now, suppose $b>0$. Suppose first that $E''$ has a $2$-torsion point over $\Q$, which implies by Theorem \ref{IsogClassification} that $p \leq 7$. Then any subgroup of $\rho_{E'',2^\infty}(\Gal_\Q)$ will be of order which is a power of $2$, from which we conclude that $\rho_{E'',2^\infty}(\Gal_\Q)=\rho_{E'',2^\infty}(\Gal_F)$. Thus all the $2$-power isogenies of $E''$ which are defined over $F$ are already defined over $\Q$. So, by \Cref{lem:isog} it follows that if $a\geq 3$, $E''$ has either an $8$-isogeny or a $4$-isogeny which is independent to a $2$-isogeny. But now it is not possible that $b>0$, as now it would follow that $E''$ has an $8p$-isogeny over $\Q$, or is isogenous over $\Q$ to an elliptic curve with an $8p$-isogeny over $\Q$. If $a=2$, then Lemma \ref{lem:isog} implies $E''/\Q$ has either a 4-isogeny or independent 2-isogenies, and thus $E''$ is isogenous over $\Q$ to an elliptic curve with a rational cyclic $4p$-isogeny. This cannot happen if $p=5,7$.

The case when $E''$ has no $2$-torsion point over $\Q$ remains.  Hence $E''$ gains a point of order 2 over a cubic field $K$. Since $\rho_{E'',2^\infty}(\Gal_K)=\rho_{E'',2^\infty}(\Gal_F)$, if follows that if $E''$ has a $2^k$-isogeny over $F$ then this same isogeny is defined over $K$. By \Cref{lem:isog} if $a >1$, then in particular $E''$ has either a $K$-rational $4$-isogeny or full $2$-torsion over $K$. If $a \geq 3$, then by Lemma \ref{lem:isog}, $E''$ has a 4-isogeny over $K$.

If $E''/\Q$ has a cyclic $4$-isogeny over $K$, it follows that $E''$ has a point of order 4 over $K$ or over a quadratic extension of $K$. By the classification of 2-adic images due to Rouse and Zureick-Brown \cite{RouseDZB} (see, in particular, the data file 2primary\textunderscore Ss.txt associated to \cite{GJLR}), this means $E''$ corresponds to a rational point on the modular curve {\href{https://users.wfu.edu/rouseja/2adic/X20.html}{$X20$}}. Recall that if $j(E'') \neq 3^3\cdot5\cdot7^5/2^7$, then $E''$ also has a rational cyclic $p$-isogeny.
\begin{enumerate}
\item If $E''$ corresponds to a rational point on the fiber product of $X20$ and $X_0(3)$, then $j(E'')=3^2\cdot 23^3/2^6$ or $-3^3\cdot 11^3/2^2$ by \cite[Prop. 6]{DanielsGJ20}. An elliptic curve with either $j$-invariant gives points on $X_0(8)$ of even degree only, so $\ord_2(\deg(\varphi)) \leq 2$. Suppose for the sake of contradiction that $E(F)$ has a point of order 8. Then some elliptic curve 2-isogenous to $E''$ must correspond to a point of odd degree on $X_1(4)$. A Magma computation shows no such curve exists; see the research website of either author.
\item The fiber product of {\href{https://users.wfu.edu/rouseja/2adic/X20.html}{$X20$}} and $X_0(5)$ has only cusps by \cite[Prop. 6k]{DanielsGJ20}.
\item If $E''$ corresponds to a rational point on the fiber product of {\href{https://users.wfu.edu/rouseja/2adic/X20.html}{$X20$}} and $X_0(7)$, then $j(E'') =-3^3 \cdot 13 \cdot 479^3/2^{14}$ or $3^3 \cdot 13/2^2$ by \cite[Prop. 6s]{DanielsGJ20}. As in the fiber product of {\href{https://users.wfu.edu/rouseja/2adic/X20.html}{$X20$}} and $X_0(3)$, these two elliptic curves give no points of odd degree on $X_0(8)$, and no elliptic curve 2-isogenous to $E''$ gives a point on $X_1(4)$ of odd degree. Thus $a \leq 2$. In this case, it is possible that $E''/\Q$ does not have a rational cyclic 7-isogeny, but then $j(E'')=3^3\cdot5\cdot7^5/2^7$. Suppose $E''$ is an elliptic curve with $j(E'')=3^3\cdot5\cdot7^5/2^7$ and let $\psi_{E''}(n)$ be its $n$-th division polynomial. We check that $\psi_{E''}(4)/\psi_{E''}(2)$, the polynomial whose roots are $x$-coordinates of points of order $4$ on $E''$ and hence generate the fields over which $E''$ has a $4$-isogeny, is an irreducible degree $6$ polynomial, so we conclude that $E''$ has no $4$-isogenies over any odd degree number field.

\item If $E''$ corresponds to a rational point on $X_0(11)$, then $j(E'')=-11\cdot 131^3$ or $-11^2$ (see, for example, \cite[Table 4]{LRAnn}). As before, we may choose a specific elliptic curve over $\Q$ with each $j$-invariant and use division polynomials to show $E''$ does not have a 4-isogeny over any odd degree number field.
\item The fiber product of {\href{https://users.wfu.edu/rouseja/2adic/X3.html}{$X3$}} and $X_0(13)$ has only cuspidal rational points by \cite[Table 8]{DanielsGJ}, and thus so does the fiber product of {\href{https://users.wfu.edu/rouseja/2adic/X20.html}{$X20$}} and $X_0(13)$ since {\href{https://users.wfu.edu/rouseja/2adic/X20.html}{$X20$}} covers {\href{https://users.wfu.edu/rouseja/2adic/X3.html}{$X3$}}.
\end{enumerate}

Suppose now $E''$ has full $2$-torsion over $K$ but no $4$-isogeny; in particular, as noted above, this can happen only if $a \leq 2$. Then $\rho_{E'',2}(\Gal_\Q)$ is cyclic and
we see that $p\neq 5,13$ by \cite[Tables 8 \& 11]{DanielsGJ} and $p\neq 11$ by checking that the curves with $11$-isogenies over $\Q$ (i.e $j(E'') \in \{-11\cdot 131^3, -11^2\}$) have surjective mod 2 representations. Hence the only remaining possibilities are $p=3,7$.  We see that $j(E'')= 3^3\cdot5\cdot7^5/2^7$ does not appear in this case as for $E$ with this $j$-invariant $\im \rho_{E,2}\simeq S_3$.

\section{Proof of Theorem \ref{thm:FiniteIsogenyClasses}}
In this section, we prove our main result: If $x \in X_1(N)$ is a point of odd degree corresponding to a non-CM $\Q$-curve $E$ and $\deg(x) < \gon_{\Q}(X_1(N))$, then $E$ belongs to the $\overline{\Q}$-isogeny class of an elliptic curve with $j$-invariant $-140625/8$. By Proposition \ref{prop:combine2} and Theorem \ref{Thm:main}, it suffices to address the case where $N=2p^b$ or $4p^b$ for an odd prime $p$. If $p>3$, the general strategy is to show that the existence of such a low degree point would force a point on $X_1(2p)$ or $X_1(4p)$ which is known not to exist. (See $\S7.1$.) This is inspired by the approach of \cite{BELOV} and \cite{OddDeg}, with the distinction being that the point on $X_1(2p)$ or $X_1(4p)$ is associated to an elliptic curve in the $\overline{\Q}$-isogeny class of $E$ rather than $E$ itself. The prime $p=3$, addressed in $\S7.2$, requires a different approach. There, we show such a low degree point on $X_1(2\cdot 3^k)$ or $X_1(4\cdot 3^k)$ would force the existence of an elliptic curve over $\Q$ with unexpected entanglement between its 2- and 27-torsion point fields. These elliptic curves correspond to rational points on one of 10 possible curves, all of which have genus at least 2. In $\S7.3$, we determine the elliptic curves which correspond to rational points on each of these 10 curves and show that none are $\overline{\Q}$-isogenous to a non-CM elliptic curve producing a point of odd degree on $X_1(N)$ with $\deg(x) < \gon_{\Q}(X_1(N))$.

\subsection{Points on $X_1(2^ap^k)$ for $p > 3$.}  In this section, we show if $p >3$ is prime, there are no sporadic points of odd degree on $X_1(2 p^k)$ associated to non-CM $\Q$-curves. We also show that there are no sporadic points of odd degree $X_1(4\cdot7^k)$ associated to a non-CM $\Q$-curve $E$.

\begin{lem}\label{lem:ReplaceIsogeny}
Let $p$ be prime. Suppose $E/F$ is a non-CM elliptic curve with $[E,C] \in X_0(p)(\Q(j(E)))$. If $Q\in E(F)$ is a point of order $p$, then there is a subfield of $F$ of degree dividing $\frac{p-1}{2}\cdot [\Q(j(E)):\Q]$ that is the residue field of a point on $X_1(p)$ associated to either $E$ or $E/C$.
\end{lem}

\begin{proof}
If $\langle Q \rangle =C$, the result follows by \cite[Thm. 5.5]{BCS}, so suppose not. Let $\varphi: E \rightarrow E' \coloneqq E/C$, viewed as an isogeny of elliptic curves over $F$. Then $\varphi(Q) \in E'(F)$ is a point of order $p$, and since $\hat{\varphi}(\varphi(Q))=[p]Q=O$ this point generates $\ker(\hat{\varphi})$. Since $[E', \ker(\hat{\varphi})] \in X_0(p)(\Q(j(E))$, the residue field of $[E', \varphi(Q)] \in X_1(p)$ has degree dividing $\frac{p-1}{2}\cdot[\Q(j(E)):\Q]$ by \cite[Thm. 5.5]{BCS} and is contained in $F$, as desired.
\end{proof}

\begin{prop}
Suppose $x \in X_1(2p^k)$ is a point of odd degree corresponding to a non-CM $\Q$-curve, where $p \geq 5$ is prime. Then $\deg(x) \geq \gon_{\Q}(X_1(2p^k))$. In particular, $x$ is not sporadic.
\end{prop}

\begin{proof}
Let $F \coloneqq \Q(x)$, and let $E/F$ be the non-CM $\Q$-curve associated to $x$ with an $F$-rational point of order $2p^k$. Suppose $\deg(x) < \gon_{\Q}(X_1(2p^k))$. Since $X_1(2)$ and $X_1(10)$ have genus 0, we may always assume $k \geq 1$ and that $k \geq 2$ if $p=5$. By \Cref{thm:CN}, there exists an $F$-rational isogeny $\varphi: E \rightarrow E'$ such that $j(E') \in \Q$. In this case, it follows from \Cref{prop:CN} that there exists a cyclic subgroup $C$ of $E'$ such that $[E',C] \in X_0(p)(\Q)$ or else $p=7$ and $j(E') = 3^3\cdot5\cdot7^5/2^7$.

For now, suppose that if $p=7$ then $j(E')\neq 3^3\cdot5\cdot7^5/2^7$.  From Proposition \ref{Prop:div}, we have $p \in \{5,7,11,13\}$ and $d_0 \mid [F:\Q]$ where
\[
d_0=\begin{cases}
5^{2k-3} \text{ if } p=5,\\
7^{2k-2} \text{ if } p=7,\\
5 \cdot 11^{2k-2} \text{ if } p=11,\\
3\cdot 13^{2k-2} \text{ if } p=13.
\end{cases}
\]
Since we have assumed $\deg(x) < \gon_{\Q}(X_1(2p^k))$ and $\deg(x)$ is odd, it must be that $[F:\Q]=d_0$. Indeed, by Proposition \ref{prop:Degree} and the $\Q$-gonality bounds on $X_1(14)$, $X_1(22)$, $X_1(26)$ and $X_1(50)$ from Table 1 in \cite{derickxVH}, we see that $\gon_{\Q}(X_1(2p^k)) \leq 3d_0$. In fact, if $p=11$, it follows that $\gon_{\Q}(X_1(2\cdot 11^k)) \leq 4 \cdot 11^{2k-2}<d_0$, so the claim follows already in this case.

Suppose for the sake of contradiction that $[F:\Q] =d_0$, and let $\varphi: E \rightarrow E'$ be the $F$-rational isogeny defined above. It follows from \Cref{cor:IsogenyCor} that $E'$ has a point $Q$ of order $p$ over a field extension $F'/F$ of degree dividing $p$. By composing $\varphi$ with the $F$-rational isogeny $E' \rightarrow E'/C$ induced by $[E',C] \in X_0(p)(\Q)$ and replacing $E'$ with $E'/C$ if necessary, we may assume by Lemma \ref{lem:ReplaceIsogeny} that the residue field of $[E',Q] \in X_1(p)$ has degree dividing $(p-1)/2$ and is contained in $F'$.

Since $E'$ also has a 2-isogeny over $F$ (because $E$ does), together these imply $E'$ has a point $P$ of order $2p$ over $F'$, a field of degree dividing $pd_0$. Let $y \coloneqq [E',P] \in X_1(2p)$, and let $y_1 \coloneqq [E',2P]=[E',Q] \in X_1(p)$ and $y_2 \coloneqq [E',pP] \in X_1(2)$ be the induced points. Since Proposition \ref{Prop:div} applies in particular to any non-CM elliptic curve with rational $j$-invariant, we see that $3 \nmid [F':\Q(y_1)]$. Thus $\Q(y)=\Q(y_1)\Q(y_2) =\Q(y_1)$. Since $y_1$ has odd degree dividing both $pd_0$ and $(p-1)/2$, then Proposition \ref{Prop:div} shows
\[
\deg(y_1)=\begin{cases}
1 \text{ if } p=5, 7\\
3 \text{ if } p=13.
\end{cases}
\]

Since $\Q(y)=\Q(y_1)$, we have an elliptic curve with rational $j$-invariant possessing a point of order $2p$ over a field of the degree indicated above. If $p=13$, this contradicts Gu\v{z}vi\'{c} \cite{Guzvic}. If $p=7$, this contradicts Mazur \cite{mazur77}. If $p=5$, then in particular this implies any elliptic curve $E''/\Q$ with $j(E'')=j(E')$ has a rational cyclic $2$-isogeny. Thus $E''$ cannot have a rational cyclic 25-isogeny or two independent 5-isogenies, and Proposition \ref{Prop:index} shows that $\ord_5([\GL_2(\Z_{5}): \im \rho_{E''/\Q, 5^{\infty}}]) =0$. Then Lemma \ref{lem:PrelimDiv} shows in fact $5^{2k-2} \mid [F:\Q]$, and we have reached a contradiction.

Finally, suppose $j(E') = 3^3\cdot5\cdot7^5/2^7$ and $p=7$. Then by Proposition \ref{Prop:div}, we have $3 \cdot 7^{\max(0,2k-3)} \mid [F:\Q]$. As above, $E'$ attains a point of order $2 \cdot 7$ in an extension $F'/F$ of degree dividing $7$. For a fixed elliptic curve $E''/\Q$ with $j(E'')= 3^3\cdot5\cdot7^5/2^7$, we may compute the 14-division polynomials to see that $3^3$ divides the degree of the residue field of any point of odd degree on $X_1(14)$ associated to $E''$. Thus $3^3 \mid [F':\Q]$, and together these imply $3^3 \cdot 7^{\max(0,2k-3)} \mid [F:\Q]$. As above, we see $\gon_{\Q}(X_1(2 \cdot 7^k))\leq 2 \cdot 7^{2k-2}<[F:\Q]$, as desired.
\end{proof}

\begin{prop} \label{Prop:4with7}
Let $x \in X_1(4\cdot7^k)$ be a point of odd degree corresponding to a non-CM $\Q$-curve $E$. Then $\deg(x) \geq \gon_{\Q}(X_1(4\cdot 7^k))$. In particular, $x$ is not sporadic.

\end{prop}

\begin{proof}
Let $x\ \in X_1(4\cdot7^k)$ be a point of odd degree corresponding to a non-CM $\Q$-curve $E$. Since $X_1(4)$ has genus 0, we may assume $k>0$. Let $F \coloneqq \Q(x)$, and fix a model of $E/F$ with an $F$-rational point of order $4 \cdot 7^k$. By \Cref{thm:CN}, there is an $F$-rational cyclic isogeny $\varphi:E \rightarrow E'$ such that $j(E') \in \Q$. Let $E''/\Q$ be an elliptic curve with $j(E'')=j(E')$. For now, suppose $j(E'') \notin \{-3^3 \cdot 13 \cdot 479^3/2^{14}, \, 3^3 \cdot 13/2^2\}$. By Corollary \ref{cor:5.3}, we may assume $E''$ has a rational cyclic 7-isogeny and full 2-torsion (but no 4-isogeny) over a cubic extension $K$. In particular, this means $j(E') \neq 3^3\cdot5\cdot7^5/2^7$.

Suppose for the sake of contradiction that $\deg(x) < \gon_{\Q}(X_1(4\cdot 7^k))$. By Proposition \ref{Prop:div}, we know that
$
7^{2k-2} \mid [F:\Q].
$
Since $E''$ corresponds to a degree 3 point on $X_1(2)$ and $E'(F)$ has a point of order 2 (since $E(F)$ does), it follows that $3 \mid [F:\Q]$. Together, this gives
\[
3\cdot 7^{2k-2} \mid [F:\Q].
\]
Since $X_1(28)$ has $\Q$-gonality 6 by \cite{derickxVH}, it follows that $\gon_{\Q}(X_1(4 \cdot 7^k) \leq 6 \cdot 7^{2k-2}$. Thus our assumption that $\deg(x) < \gon_{\Q}(X_1(4\cdot 7^k))$ and is of odd degree implies
\[
\deg(x)=[F:\Q]=3\cdot 7^{2k-2}.
\]

As in the proof of Proposition \ref{prop:combine2}, we have $\rho_{E'',2^{\infty}}(\Gal_K)=\rho_{E'',2^{\infty}}(\Gal_F)$, and in particular $E''$ has no 4-isogeny over $F$. Thus $E'/F$ has no 4-isogeny. It follows that $\ord_2(\deg(\varphi))=1$. If $\ker(\varphi)=\langle P \rangle$, we see that $\varphi$ factors over $F$ as
\[
E \xrightarrow{\varphi_1} E_0 \coloneqq E/\langle 2P \rangle \xrightarrow{\varphi_2} E'.
\]
 Since $2 \nmid \deg(\varphi_1)$, there is a point of order 4 in $E_0(F)$, since $E(F)$ has such a point. By Corollary \ref{cor:IsogenyCor}, the curve $E'$ has a rational point of order 7 in an extension of $F'/F$ of degree dividing $7$. Since $E'/F'$ has a point of order 7 and $\deg(\hat{\varphi_2})=2$, it follows that $E_0$ has a point of order 7---and hence a point of order 28---over $F'$. That is, there exists $a\coloneqq [E_0,Q] \in X_1(28)$ of odd degree dividing $3 \cdot 7^{2k-1}$.

We note that $[\Q(j(E_0)):\Q]=3$, since $\hat{\varphi_2}$ has degree 2 and any point on $X_0(2)$ associated to $E'$ has degree 3. In fact, $\Q(\hat{\varphi_2})=\Q(j(E_0))=K$ for the cubic extension $K$ defined above. Moreover, any $E_1/K$ with $j(E_1)=j(E_0)$ will have a rational cyclic 7-isogeny and rational point of order 2, since $E''$ does. Thus $E_0$ corresponds to a point of degree 3 on $X_1(2)$, and this is the only possible degree of a point on $X_1(2)$ associated to $E_0$ having odd degree. It follows that $a_1 \coloneqq [E_0, 7Q] \in X_1(4)$ has degree 3 as well since it has odd degree and its degree over $[E_0, 14Q] \in X_1(2)$ divides $2^2$ by Proposition 4.6 in \cite{GJNajman}. Since there exists $[E_0,C] \in X_0(7)(\Q(j(E_0)))$, by replacing $E_0$ with $E_0/C$ if necessary we may assume by Lemma \ref{lem:ReplaceIsogeny} that $a_2 \coloneqq [E_0,4Q] \in X_1(7)$ has degree dividing $3 \cdot \frac{7-1}{2}=9$. Thus it must have degree 3. Since $\Q(a_1)\Q(a_2)$ is contained in $\Q(a)$, it must be that $\Q(a_1)\Q(a_2)$ has degree 3. As $\Q(a)$ is at most a quadratic extension of this field, it follows that $\Q(a)$ has degree 3. But this contradicts Theorem A in \cite{Deg3Class}.

Finally, assume $j(E'') \in \{-3^3 \cdot 13 \cdot 479^3/2^{14}, \, 3^3 \cdot 13/2^2\}$, and let $K_0$ denote the residue field of the degree 3 point on $X_1(2)$ associated to $E''$. As noted above, $K_0 \subseteq F$, and by Corollary \ref{cor:IsogenyCor}, $E'$ must have a point of order $7$ over $F(\zeta_7)$. A computation in Magma shows the $7$-division polynomial of $E''$ over $K_0(\zeta_7)$ factors as a product of irreducible polynomials of degree 3 and 21, which shows that any $F$ over which $E'$ is isogenous to an elliptic curve with a point of order 28 is divisible by 9. Thus by \Cref{Prop:div}, $F$ is divisible by $9 \cdot 7^{2k-2}$ and the claim follows.
\end{proof}

\subsection{Points on $X_1(2^a \cdot 3^k)$} Let $x\in X_1(2^a3^k)$ be a point of odd degree associated to a non-CM $\Q$-curve $E$. In this section, we show that $\deg(x) < \gon_{\Q}(X_1(2^a3^k))$ would imply $E$ is $\overline{\Q}$-isogenous to an elliptic curve corresponding to a rational point on one of 10 curves of genus at least 2.

\begin{lem} \label{lem:entanglement}
Let $E$ be an elliptic curve over $\Q$, and let $p$ be an odd prime. If $p>3$ and $E$ has no point of order $2$ over $\Q(E[p])$, then $E$ has no point of order $2$ over $\Q(E[p^k])$ for any $k \in \Z^+$. If $p=3$ and $E$ has no point of order $2$ over $\Q(E[3^{d+1}])$, where $3^d$ is the level of $\rho_{E,3^{\infty}}$, then $E$ has no point of order $2$ over $\Q(E[3^k])$ for any $k \in \Z^+$.
\end{lem}

\begin{proof}
We follow the approach of Proposition 6.1 in \cite{BELOV} and Lemma 19 in \cite{OddDeg}. For $s \in \Z^+$, let
\begin{align*}
L_s \coloneqq \ker(\im \rho_{E,2\cdot p^s} \rightarrow \im \rho_{E,p^s}),\\
K_s \coloneqq \ker(\im \rho_{E,2\cdot p^s} \rightarrow \im \rho_{E,2}),\\
K \coloneqq \ker(\im \rho_{E,2\cdot p^{\infty}} \rightarrow \im \rho_{E,2}).
\end{align*}
We may view $L_s$ as a subgroup of $\im \rho_{E, 2}$ and $K_s$ as a subgroup of $\im \rho_{E, p^s}$. Moreover, by Goursat's Lemma (see e.g., \cite[pg75]{Lang-algebra} or \cite{Goursat89}) we have the following diagram:
        \begin{equation*}\label{eq:diag2}
        \xymatrix{
            \im \rho_{E, p^{s}}/K_{s} \ar@{->>}[r] \ar[d]^{\cong} &  \im \rho_{E, p}/K_{1} \ar[d]^{\cong} \\
            \im \rho_{E, 2} / L_{s} \ar@{->>}[r]
             & \im \rho_{E, 2} / L_{1}}
        \end{equation*}
        The order of the kernel of the top map is a power of $p$, and so the order of the kernel of the bottom map is as well. Thus $[L_1:L_{s}]$ is a power of $p$, and more generally $[L_{s_1}:L_{s_2}]$ is a power of $p$ for all $ 1 \leq s_1 \leq s_2$. If $p>3$, it follows that for all $s \in \Z^+$ we have $L_1=L_s$ and
        \[
        \Q(E[2]) \cap \Q(E[p^s])=\Q(E[2]) \cap \Q(E[p]).
        \] In particular, if $E$ has no point of order 2 over $\Q(E[p])$, then $E$ has no point of order 2 over $\Q(E[p^k])$ for any $k$, as desired. Thus we may henceforth assume $p=3$.

Suppose $E$ has no point of order 2 over $\Q(E[3^{d+1}])$. The $x$-coordinate of any 2-torsion point generates a degree 3 extension of $\Q(E[2]) \cap \Q(E[3^{d+1}])$ contained in $\Q(E[2])$, from which is follows that $3 \mid \#L_{d+1}$. Similarly, $3 \mid \#L_d$. Since $\# L_i \mid 6$ for all $i$ and $[L_d:L_{d+1}]$ is a power of $3$, it follows that  $L_d=L_{d+1}$. Again by Goursat's Lemma
        \begin{equation*}
        \xymatrix{
            \im \rho_{E, 3^{d+1}}/K_{d+1} \ar@{->>}[r] \ar[d]^{\cong} &  \im \rho_{E, 3^d}/K_{d} \ar[d]^{\cong} \\
            \im \rho_{E, 2} / L_{d+1} \ar@{=}[r]
             & \im \rho_{E, 2} / L_{d}}
        \end{equation*}
      By assumption, $ \im \rho_{E, 3^{d+1}}=\pi^{-1}( \im \rho_{E, 3^d})$, and so this diagram implies
\[
\ker(K \text{ mod }{3^{d+1}} \rightarrow K \text{ mod }{3^d})=I+\text{M}_2(3^d\Z/3^{d+1}\Z).
\]
By Proposition 3.5 of \cite{BELOV},
\[
\ker(K \rightarrow K \text{ mod }{3^d})=I+3^d \text{M}_2(\Z_3),
\]
and so $\im \rho_{E,2\cdot 3^{\infty}} = \pi^{-1}(\im \rho_{E,2\cdot 3^d})$. This means $E$ has no point of order 2 over $\Q(E[3^k])$ for any $k \in \Z^+$.
\end{proof}

\begin{lem} \label{lem:ResidueFieldContainment}
Let $p$ be prime and let $F$ be a number field of odd degree. Suppose $E/F$ is a non-CM elliptic curve with $P \in E(F)$ of order $p^k$, and let $\varphi:E \rightarrow E'$ be an $F$-rational isogeny with $\deg(\varphi)=p^a$ and $j(E') \in \Q$. If $x=[E,P] \in X_1(p^k)$, then for any $E''/\Q$ with $j(E'')=j(E')$, we have
\[
\Q(x) \subseteq \Q(E''[p^{\infty}]).
\]
\end{lem}

\begin{proof}
Let $\varphi:E \rightarrow E'$ be the $F$-rational isogeny as in the lemma statement. By considering the dual isogeny $\hat{\varphi}: E' \rightarrow E$, we may view $E$ as the quotient of $E'$ by some cyclic subgroup $C$ of order $p^a$. By Lemma \ref{lem:IsogenyResidueField},
\[
F' \coloneqq \Q(j(E),j(E'))=\Q(j(E))
\]
is the residue field of $[E',C] \in X_0(p^a)$. Thus, by replacing $E,E'$ with elliptic curves over $F'$ having the same $j$-invariant if necessary (which is permissible since $\Q(x)$ is model-independent),  we may view $\hat{\varphi}: E' \rightarrow E$ as a cyclic isogeny of elliptic curves over $F'$. It follows that
\[
F'(E'[p^{\infty}])=F'(E[p^\infty]).
\]
Since $\Q(x)$ is an extension of $F'$ contained in $F'(E[p^\infty])$, we have that $\Q(x)\subseteq F'(E'[p^\infty])$. Now, let $E''$ be an elliptic curve over $\Q$ with $j(E'')=j(E')$. Then there exists a quadratic extension $L/F'$ over which $E'$ and $E''$ become isomorphic. Thus
\[
\Q(x) \subseteq L(E''[p^{\infty}])=L(E'[p^{\infty}]).
\]
But since $F' \subseteq \Q(E''[p^{\infty}])$, it follows that $L(E''[p^{\infty}])$ is at most a degree 2 extension of $\Q(E''[p^{\infty}])$. Since $\Q(E''[p^{\infty}])$ is a Galois extension and $[\Q(x):\Q]$ is odd, it must be that $\Q(x) \subseteq \Q(E''[p^{\infty}])$.
\end{proof}

\begin{lem} \label{lem:SurjectiveMod2}
Suppose $E/\Q$ is an elliptic curve with $\im \rho_{E,3^{\infty}}$ contained in one of the following groups:
\begin{align*}
9H^0\mhyphen9a, \, \, \, & 9H^0\mhyphen9b, \, \, \, 9H^0\mhyphen9c\\
9J^0\mhyphen9a, \, \, \, & 9J^0\mhyphen9b, \, \, \, 9J^0\mhyphen9c
\end{align*}
Then the mod 2 Galois representation associated to $E$ is surjective.
\end{lem}

\begin{proof}
First, note that $E$ cannot have a rational point of order 2. The fiber product of $9C^0 \mhyphen 9a$ with $X_0(2)$ has no non-CM, non-cuspidal points by \cite[proof of Prop. 23]{OddDeg}, so it remains to check the fiber product of $X_0(2)$ with each of the first 3 groups listed (since the latter 3 all cover $9C^0 \mhyphen 9a$).

Suppose $\Q(E[2])$ is a cyclic cubic extension of $\Q$. Then $\im \rho_{E, 3^{\infty}}$ cannot be contained in $3D^0\mhyphen3a$ as an elliptic curve cannot have simultaneously this mod 3 image and a cyclic mod 2 image of order 3 by \cite[Tables 8 \& 11]{DanielsGJ}. Thus since the first 3 groups cover $3D^0\mhyphen3a$, we may assume $\im \rho_{E, 3^{\infty}}$  is contained in
\[
9J^0\mhyphen9a, \, \, \,  9J^0\mhyphen9b, \, \, \, 9J^0\mhyphen9c.
\]
For each group we compute the fiber product (over $X_0(1)$) of the modular curve parameterizing elliptic curves $E$ with cyclic $\im \rho_{E,2}$ with the modular curve parameterizing elliptic curves with mod $9$ image contained in the respective groups. In all the cases we get genus 2 curves with Jacobian of rank 0 over $\Q$ and an easy computation shows that none of the modular curves have any non-cuspidal rational points. See the research website of either author for the code used.
\end{proof}

\begin{prop}\label{prop:2and3}
Let $x=[E,P] \in X_1(2\cdot 3^k)$ be a point of odd degree, where $E$ is a non-CM $\Q$-curve. If $\deg(x)<\gon_{\Q}(X_1(2\cdot 3^k))$, then $E$ is $\overline{\Q}$-isogenous to an elliptic curve corresponding to a rational point on $X_{H_i}$ for $1 \leq i \leq 6$, where $H_i$ is as defined in the appendix. In particular, since all curves $X_{H_i}$ have genus at least 2, $x$ is not sporadic unless $E$ is $\overline{\Q}$-isogenous to one of a finite number of elliptic curves over $\Q$.
\end{prop}

\begin{proof}
Let $F \coloneqq \Q(x)$, and fix a model of $E/F$ with $P \in E(F)$. We may assume $k \geq 2$ since $X_1(6)$ has genus 0. Since $X_1(18)$ has $\Q$-gonality 2 by \cite{derickxVH}, it follows that $\gon_{\Q}(X_1(2 \cdot 3^k))$ is at most
\[
\deg(X_1(2\cdot 3^k) \rightarrow X_1(18))\cdot 2=3^{2k-4}\cdot 2.
\]
Since we have assumed $\deg(x)<\gon_{\Q}(X_1(2\cdot 3^k))$ and  $3^{2k-4} \mid [F:\Q]$ by Proposition \ref{Prop:div}, it must be that $[F:\Q] = 3^{2k-4}$. Note in this case Proposition \ref{Prop:div} also implies $F=\Q(f(x))$, where $f: X_1(2\cdot 3^k) \rightarrow X_1(3^k)$.

By \Cref{thm:CN}, there exists an $F$-rational isogeny $\varphi: E \rightarrow E'$ such that $j(E') \in \Q$. We may assume $\varphi$ is cyclic and generated by a point $Q$ of order $3^a\cdot n$ where $3 \nmid n$. If $n >1$, then we replace $E$ with the quotient $E/ \langle 3^a Q \rangle$. This new curve has a point of order $3^k$, and also a rational cyclic 2-isogeny (since $E$ has one). Thus, by replacing $E$ by this quotient if necessary, we may assume $\varphi$ has degree $3^a$.

Now, let $E''$ be an elliptic curve over $\Q$ with $j(E'')=j(E')$. By Lemma \ref{lem:ResidueFieldContainment}, we have $F \subseteq \Q(E''[3^{\infty}])$. Also, since $E$ has a rational 2-isogeny over $F$, so does $E''$. That is, there exists a point $P_0 \in E''$ of order 2 such that the residue field of $[E'', P_0] \in X_1(2)$ is contained in $\Q(E''[3^{\infty}])$.

It follows from Lemma \ref{lem:PrelimDiv} that
\[
3^{\text{max}(0,2k-2-d)} \mid [F:\Q]
\]
where $d=\ord_3([\GL_2(\Z_{3}): \im \rho_{E''/\Q, 3^{\infty}}])$. Since $[F:\Q] = 3^{2k-4}$, it must be that $d \geq 2$. Thus $d=2$ by Proposition \ref{Prop:index}, and $E''$ must have a rational cyclic 3-isogeny over $\Q$ by Proposition 3.3 in \cite{CremonaNajmanQCurve}. Thus $\im \rho_{E'', 3^{\infty}}$ is contained in one of the following by Appendix A of \cite{OddDeg} (see also index data in Table 1 of \cite{SutherlandZywina}):
\begin{align*}
9H^0\mhyphen9a, \, \, \, & 9H^0\mhyphen9b, \, \, \, 9H^0\mhyphen9c\\
9I^0\mhyphen9a, \, \, \, & 9I^0\mhyphen9b, \, \, \, 9I^0\mhyphen9c\\
9J^0\mhyphen9a, \, \, \, & 9J^0\mhyphen9b, \, \, \, 9J^0\mhyphen9c\\
& 27A^0\mhyphen27a
\end{align*}

Since it suffices to identify a single elliptic curve over $\Q$ in the $\overline{\Q}$-isogeny class of $E$, we are free to replace $E''$ with a different elliptic curve in its $\Q$-isogeny class. In particular, since any elliptic curve with a rational cyclic isogeny of degree 9 is isogenous to an elliptic curve with two independent 3-isogenies, we may assume $E''$ does \emph{not} have a rational cyclic 9-isogeny. That is, we may assume $\im \rho_{E'', 3^{\infty}}$ is contained in one of the following:
\begin{align*}
9H^0\mhyphen9a, \, \, \, & 9H^0\mhyphen9b, \, \, \, 9H^0\mhyphen9c\\
9J^0\mhyphen9a, \, \, \, & 9J^0\mhyphen9b, \, \, \, 9J^0\mhyphen9c
\end{align*}
Note that it is still the case that $E''$ has a point $P_0$ of order 2 defined over $\Q(E''[3^{\infty}])$.

First, note that the mod 2 Galois representation associated to $E''$ is surjective by Lemma \ref{lem:SurjectiveMod2}. Thus $\Q(P_0)$ is a cubic extension contained in $\Q(E''[3^{\infty}])$. Since the latter extension is Galois, it follows that $\Q(E''[2]) \subseteq \Q(E''[3^{\infty}])$. By Lemma  \ref{lem:entanglement}, we have $\Q(E''[2]) \subseteq \Q(E''[3^{3}])$. We use Magma to search for subgroups of $\GL_2(\Z/54\Z)$ which could be the image of the mod 54 Galois representation associated to $E''$, following the method of \cite[Prop. 24]{OddDeg}. This results in a list of 30 possible subgroups (up to conjugacy), but for each the mod 18 reduction is (conjugate to) a subgroup of one of the groups $H_i$ for $1 \leq i \leq 6$; see the appendix for a list of generators for each $H_i$. See the research website of either author for the code used. In each case, the genus of $X_{H_i}$ is at least 2. \qedhere
\end{proof}

\begin{prop} \label{prop:4and3}
Let $x \in X_1(4\cdot 3^k)$  be a point of odd degree corresponding to a non-CM $\Q$-curve $E$. If $\deg(x) < \gon_{\Q}(X_1(2\cdot 3^k))$, then $E$ is $\overline{\Q}$-isogenous to an elliptic curve corresponding to a rational point on $X_{H_i}$ for $7 \leq i \leq 10$, where $H_i$ is as defined in the appendix. In particular, since all curves $X_{H_i}$ have genus at least $2$, $x$ is not sporadic unless $E$ is $\overline{\Q}$-isogenous to one of a finite number of elliptic curves over $\Q$.
\end{prop}

\begin{proof} Let $x\ \in X_1(4\cdot 3^k)$ be a point of odd degree corresponding to a non-CM $\Q$-curve $E$, and suppose $\deg(x)<\gon_{\Q}(X_1(2\cdot 3^k))$. Since $X_1(4)$ and $X_1(12)$ have genus 0, we may assume $k>1$. Let $F \coloneqq \Q(x)$, and fix a model of $E/F$ with an $F$-rational point of order $4 \cdot 3^k$. By \Cref{thm:CN}, there is an $F$-rational cyclic isogeny $\varphi:E \rightarrow E'$ such that $j(E') \in \Q$. Suppose first that $j(E') \in \{3^2 \cdot 23^3/2^6, \, -3^3\cdot 11^3/2^2\}$. Since these elliptic curves lie in the same isogeny class, it suffices to consider only one of them, so suppose $j(E')=-3^3\cdot 11^3/2^2$. Let $E''/\Q$ be an elliptic curve with $j(E'')=j(E')=-3^3\cdot 11^3/2^2$. We check that $3$-adic image (with $\pm I$ added) is $3B^0\mhyphen 3A$; we do this by showing that $j_G(t)=j(E'')$ has no roots for all minimal subgroups $G$ of $3B^0 \mhyphen 3A$. Hence
$$\ord_3([\GL_2(\Z_{3}): \im \rho_{E''/\Q, 3^{\infty}}])=\ord_3([\GL_2(\Z/3\Z): \im \rho_{E''/\Q, 3}])=0.$$
Thus by  \Cref{lem:PrelimDiv},
\[
3^{2k-2} \mid [F:\Q].
\] Since $X_1(36)$ has $\Q$-gonality 8 by \cite{derickxVH}, it follows that $\gon_{\Q}(X_1(4 \cdot 3^k)) \leq 8 \cdot 3^{2k-4}<9 \cdot 3^{2k-4}=3^{2k-2}$, and we have a contradiction.

Henceforth we may assume $j(E') \notin \{3^2 \cdot 23^3/2^6, \, -3^3\cdot 11^3/2^2\}$, and so by Corollary \ref{cor:5.3}, any $E''/\Q$ with $j(E'')=j(E')$ has a rational cyclic 3-isogeny and either (a) a rational point of order 2 or (b) full 2-torsion over a cubic extension but no 4-isogeny over this extension.

First, suppose $E''$ has a rational point of order 2. Then as in the third paragraph of the proof of Proposition \ref{prop:combine2}, $E''$ has either a rational cyclic 4-isogeny or independent 2-isogenies. In particular, this means that $E''$ cannot correspond to a rational point on $X_{3D^0\mhyphen3a}=X_0(3,3)$, $X_{9B^0\mhyphen9a}=X_0(9)$ by Kenku \cite{kenku}. In addition, $E''$ cannot correspond to a rational point on $X_{9C^0\mhyphen9a}$ since the fiber product of this curve with $X_0(2)$ has no non-cuspidal, non-CM rational points, as in the proof of Proposition 23 in \cite{OddDeg}. Thus we may assume $E''$ has 3-adic image $3B^0\mhyphen3a$, and we reach a contradiction as before.

Now we suppose $E''$ has full 2-torsion over a cubic extension but no 4-isogeny over this cubic field. Recall that if $E''$ has full 2-torsion over a cubic extension, then its discriminant is a square and it corresponds to a rational point on $X_{2A^0\mhyphen 2a}$. As shown in Proposition 24 of \cite{OddDeg}, this implies $E''$ cannot have a 9-isogeny or two independent 3-isogenies. We have already shown $E''$ cannot have 3-adic image $3B^0\mhyphen3a$, so it must be that $E''$ corresponds to a rational point on the modular curve $X_{9C^0\mhyphen9a}$. By Lemma \ref{lem:SurjectiveMod2}, we may assume the 3-adic image of $E''$ is $9C^0\mhyphen9a$.

Since $\ord_3([\GL_2(\Z_{3}): \im \rho_{E''/\Q, 3^{\infty}}])=1$, by Lemma \ref{lem:PrelimDiv}, we know that
$
3^{2k-3} \mid [F:\Q].
$
Since, as we have already noted, $\gon_{\Q}(X_1(4 \cdot 3^k)) \leq 8 \cdot 3^{2k-4}$, it must be that $[F:\Q]=3^{2k-3}$. Let $\varphi: E \rightarrow E'$ as defined above. We may assume $\varphi$ is cyclic and generated by a point $Q$ of order $3^a\cdot n$ where $3 \nmid n$. The isogeny $\varphi$ factors over $F$ as
\[
E \xrightarrow{\varphi_1} E_0 \coloneqq E/\langle 3^aQ \rangle \xrightarrow{\varphi_2} E'.
\]
The curve $E_0$ has an $F$-rational point of order $3^k$, and also a rational cyclic 2-isogeny (since $E$ has one). That is, it has an $F$-rational point $P_0$ of order $2\cdot 3^k$. Moreover, by Lemma \ref{lem:PrelimDiv}, $3^{2k-3}$ divides the residue field of $x'' \coloneqq [E_0,P_0] \in X_1(2\cdot 3^k)$. Since $[F:\Q]=3^{2k-3}$, we may conclude $\Q(x'')$ has degree exactly $3^{2k-3}$ and $\Q(x'')=F$. In addition, $F=\Q(f(x''))$, where $f: X_1(2\cdot 3^k) \rightarrow X_1(3^k)$.

By applying Lemma \ref{lem:ResidueFieldContainment} to $\varphi_2:E_0 \rightarrow E'$, we see that $F \subseteq \Q(E''[3^{\infty}])$. Also, since $E'$ has a rational 2-isogeny over $F$, so does $E''$. That is, there exists a point $P'' \in E''$ of order 2 such that the residue field of $[E'', P''] \in X_1(2)$ is contained in $\Q(E''[3^{\infty}])$. Since the 3-adic image of $E''$ is $9C^0\mhyphen9a$, by Lemma \ref{lem:entanglement} we have $\Q(E''[2]) \subseteq \Q(E''[3^{3}])$. A Magma computation following the method of \cite[Prop. 24]{OddDeg} shows that there are 4 possible subgroups of $\GL_2(\Z/18\Z)$---up to conjugacy---that could be the image of the mod 18 Galois representation associated to $E''$. These correspond to the groups $H_7$, $H_8$, $H_9$, or $H_{10}$ as defined in the appendix. In each case, $X_{H_i}$ has genus at least 2. See the research website of the either author for the code used.
\end{proof}

\subsection{Points on modular curves}
In this section, we find the rational points on the curves $X_{H_i}$, as in Propositions \ref{prop:2and3} and \ref{prop:4and3}. Recall each $H_i$ is a subgroup of $\GL_2(\Z/18\Z)$, as defined in the appendix. There are two non-CM $j$-invariants associated to rational points on these curves (see Lemma \ref{lem:NewPoints}), and the corresponding elliptic curves lie in a single isogeny class. In Proposition \ref{prop:1792}, we show there are no points of odd degree on $X_1(4\cdot 3^k)$ associated to curves in this isogeny class, completing the proof of Theorem \ref{thm:FiniteIsogenyClasses}.

\begin{lem}
All the non-cuspidal rational points on the modular curves $X_{H_i}$ for $i=1,2$ and $3$ correspond to elliptic curves with $j=0$.
\end{lem}
\begin{proof}
Let $E$ be an elliptic curve corresponding to a point on one of the modular curves $X_{H_i}$. There exists a modular group $H$ of level 6 and genus 1 in which all $H_i$ are of index 3. We compute that $H$ modulo $3$ is ${3D^0\mhyphen3a}$ so $j(E)=f_3(t)$, {and using Table 1 in \cite{SutherlandZywina}} we compute that
$$f_3(t)=\frac{(t(t+6)(t^2-6t+36))^3}{((t-3)(t^2+3t+9))^3}.$$
 Furthermore, $H\cap \SL(2,\Z/6\Z)$ modulo $2$ is cyclic of order $3$, which means that $\Delta(E)$ is a square in $\Q(\sqrt{-3})$. Let $E_{f_3(t)}$ be an elliptic curve with $j$-invariant $f_3(t)$. We compute that $\Delta(E_{f_3(t)})$ is a square if and only if $(t-3)(t^2+3t+9)$ is a square. Hence, we need to compute all the $\Q(\sqrt{-3})$-rational points on the elliptic curve
$$X:y^2=(t-3)(t^2+3t+9).$$
We get
$$t\in\left\{0,3,-6,3\pm 3\sqrt{-3},\frac{-3\pm 3\sqrt{-3}}{2}\right\},$$
from which we get $4$ cusps and $8$ points for which $j(E_{f_3(t)})=0$. {See the research website of either author for the code used.}
\end{proof}

\begin{lem}
        All the rational points on $X_{H_i}$ for $i=4,5$ and $6$ are cusps.
\end{lem}
\begin{proof}
First we notice that the intersection of the groups $H_i$ for $i=4,5$ and $6$ with $\{A\in \GL_2(\Z/18\Z): \det A=\pm 1\}$ are conjugate, which implies that the corresponding modular curves are isomorphic over $\Q(\zeta_9)^+$. Hence it is enough to find all the $\Q(\zeta_9)^+$-rational points on one of the curves and show that all of them are cusps.

Using Zywina's algorithm we compute a model of $X:=X_{H_4}$ to be
$$X:y^2=x^5 + 9x^4 + 29x^3 + 42x^2 + 27x + 6.$$

Let $w$ be a root of $x^3 - 6x^2 + 9x - 3$; it generates the extension $\Q(\zeta_9)^+$. To compute $X(\Q(\zeta_9)^+)$ we note that $\Aut_{\Q(\zeta_9)^+}X\simeq D_6$, and find an automorphism $a$ such that $X/a\simeq E$, where
$$E: y^2 = x^3 + (8w^2 - 37w + 33)x^2 + (22w^2 -
    126w + 155)x + (44w^2 - 248w + 300).$$
We have $E(\Q(\zeta_9)^+)\simeq \Z/2\Z \times \Z/6\Z$ and compute $\#X(\Q(\zeta_9)^+)=6$. The curve $X$ has $6$ cusps, of which 3 are rational. Since all the cusps of $X$ are necessarily defined over some subfield of $\Q(\zeta_9)$, it follows that the non-rational cusps form one orbit of $\Q(\zeta_9)^+$-rational cusps. Hence all 6 $\Q(\zeta_9)^+$-rational points of $X$ are cusps. {See the research website of either author for the code used.}
\end{proof}

\begin{lem}
The modular curves $X_{H_7}$ has no rational points.
\end{lem}
\begin{proof}
Using Zywina's algorithm we find that the non-hyperelliptic genus 4 modular $X:=X_{H_7}$ is given by the equations
$$-yz + xw=0 \text{ and }x^3 - 6x^2y + 3xy^2 + y^3 - 9z^3 + 27zw^2 - 9w^3=0$$
in $\PP^3$.
We find that $X$ has $6$ automorphisms over $\Q$, and the quotient by one of them is the genus 2 hyperelliptic curve
$$C:y^2 + (x^3 + x^2 + 1)y = 8x^6 - 23x^5 + 47x^4
    - 65x^3 + 58x^2 - 27x + 8,$$
with Jacobian of rank 0 over $\Q$. Standard techniques easily show that $C$ and hence also $X$ has no $\Q$-rational points. {See the research website of either author for the code used.}
\end{proof}

\begin{lem}\label{lem:NewPoints}
The only rational points on the modular curves $X_{H_i}$ for $i=8,9$ and $10$ correspond to elliptic curves with $j=1792$ and $406749952$
\end{lem}
\begin{proof}
First we notice that the intersection of the groups $H_i$ for $i=8,9$ and $10$ with $\{A\in \GL_2(\Z/18\Z): \det A=\pm 1\}$ are conjugate, which implies that the corresponding modular curves are isomorphic over $\Q(\zeta_9)^+$. Hence it is enough to find all the $\Q(\zeta_9)^+$-rational points on one of the curves $X$ and the $j$-invariants to which these points correspond. Using Box's modification of Zywina's algorithm we obtain a model for $X\coloneqq X_{H_8}$
$$X:y^2=x^6 - 8x^5 + 30x^4 - 70x^3 + 105x^2 - 90x + 33.$$
Since the groups $H_i$ are of index 72, the $j$-map is of degree 72, which is too large to compute directly. Hence we compute it by using the fact that $H_i$ modulo 9 is $9C^0\mhyphen9a$, and the $j$-map $f'$ (of degree 12) from the modular curve $X'$ corresponding to $9C^0\mhyphen9a$ of genus 0 to $X_0(1)$ is known by \cite{SutherlandZywina}. We compute the {H}auptmodul $h$ of the function field of $X'$ and compute the map $f:X\rightarrow X'$ such that $f(x,y)=h(q)$, which implies that $g=f'\circ f$ is the $j$-map $g:X\rightarrow X_0(1)$.

Let $w$ be a root of $x^3 - 6x^2 + 9x - 3$; it generates the extension $\Q(\zeta_9)^+$. To compute $X(\Q(\zeta_9)^+)$ we note that $\Aut_{\Q(\zeta_9)^+}X\simeq D_6$, and find an automorphism $a$ such that $X/a\simeq E$, where
$$E: y^2 = x^3 + (5w^2 - 25w + 30)x^2 + (20w^2 -
    108w + 124)/3.$$
We have $E(\Q(\zeta_9)^+)\simeq \Z/21\Z$ and compute $\#X(\Q(\zeta_9)^+)=12$. We get that $6$ of the points are cusps, and there are 3 points corresponding to each of the $j$-invariants
$1792$ and $406749952$. {See the research website of either author for the code used.}
\end{proof}

\begin{prop}\label{prop:1792}
There does not exist a $\Q$-curve $E$ corresponding to a point of odd degree on $X_1(4\cdot 3^k)$ for $k \geq 1$ which is isogenous to an elliptic curve $E'$ with $j(E') \in \{1792,406749952 \}$.
\end{prop}

\begin{proof}
We will prove the result by showing that there does not exist a number field $F$ of odd degree and a $\Q$-curve $E$ which is $F$-isogenous to an elliptic curve $E'$ with $j(E')\in \{1792,406749952 \}$ and such that $E(F)$ has a point of order 12. Since these elliptic curves are in the same isogeny class, it suffices to consider only one of them.

Suppose $E$ is a $\Q$-curve defined over a number field $F$ of odd degree with a point $P \in E(F)$ of order 12. Then by Theorem \ref{thm:CN}, there exists a $F$-rational isogeny $\varphi:E \rightarrow E'$ with $j(E') \in \Q$, and we may assume $\varphi$ is cyclic with $\ker(\varphi)=\langle Q \rangle$. Suppose for the sake of contradiction that $j(E')=1792$. Since any point on $X_0(4)$ associated to $E'$ has even degree, it must be that $\deg(\varphi)=2\cdot d$ for some odd integer $d$. Thus the image of $3P$ on $E_0 \coloneqq E/\langle 2Q \rangle$ has order 4, and by Corollary \ref{cor:IsogenyCor} there exists an extension $F'/F$ of degree dividing 3 such that $E_0(F')$ has a point of order 3. In particular, $E_0$ has a point of order 12 defined over a number field of odd degree, where $E_0$ is an elliptic curve 2-isogenous to $E'$. However, a computation in Magma shows that each elliptic curve 2-isogenous to $E'$ corresponds to points on $X_1(12)$ of even degree, contradiction.
\end{proof}

\section{Sporadic $j$-invariants in a fixed isogeny class}

Since non-CM $\Q$-curves giving rise to sporadic points of odd degree on $X_1(N)$ belong to finitely many isogeny classes, it is natural to wonder whether there can be infinitely many distinct sporadic $j$-invariants within a fixed isogeny class. The following result shows that the answer is yes, provided the degree of the sporadic point is low enough. It is worth pointing out that no known examples of non-CM sporadic points satisfy this bound, so a new idea is needed to provide a definitive answer.

\begin{proposition} \label{Prop:8.1}
Suppose that there is a non-CM point $x=[E,P] \in X_1(N)$ with
$$\deg (x)< \frac{7}{1600}[\PSL_2(\Z) : \Gamma _1(N)].$$
Then there exists infinitely many sporadic points $x'=[E',P']$ on the curves $X_1(dN)$ (with $d$ varying), such that $E'$ is isogenous to $E$ and that all the $j(E')$ are pairwise distinct. Moreover, if $\deg(x)$ is odd, we can obtain infinitely many sporadic points $x'=[E',P']$ on the curves $X_1(dN)$ such that $\deg(x')$ is odd and such that $E'$ is isogenous to $E$ and that all the $j(E')$ are pairwise distinct.
\end{proposition}

\begin{proof} Note in particular that $N$ must be greater than 2, and let $p$ be a prime divisor of $N$. Starting with a sporadic point $x=[E,P]\in X_1(N)$, we will construct a sporadic point $x'=[E', P']\in X_1(pN)$ such that
$$\deg(x') \leq p^2\deg(x)= \deg (X_1(pN)\rightarrow X_1(N))\deg (x)<\frac{7}{1600}[\PSL_2(\Z) : \Gamma _1(pN)],$$
which will show that $x'$ is sporadic. Let $F$ be the residue field of $x$, let $k=v_p(N)$ and $P_0=(N/p^k)P$.
 We have $$\im \rho_{E/F, p}\subseteq \left\{\begin{pmatrix}
                                1 & * \\
                                0 & \chi_p(\Gal_F)
                              \end{pmatrix}\right\}.$$
Let $S$ be the set of subgroups of $E[p]$ of order $p$ that are distinct from $\langle
p^{k-1}P_0\rangle$. We have $\#S=p$ and $\Gal_F$ (and hence $\im \rho_{E/F, p}$) acts on $S$.
The group $\im \rho_{E/F, p}$ is isomorphic to a subgroup of $\AGL_1(\F_p)$  and by \cite[Lemma 3.5]{CremonaNajmanQCurve} it acts on $S$ either transitively or with a fixed point. If this action has a fixed point, $E$ has an isogeny of degree $p$ {with kernel different from $\langle
p^{k-1}P_0\rangle$} over $F$. If the action is transitive, the stabilizer of a group is of index $p$ in $\Gal_F$ and hence $E$ has an isogeny of degree $p$ with kernel different from $\langle
p^{k-1}P_0\rangle$ over $F'$, an extension of $F$ of degree $p$. To conclude, there exists an isogeny $\psi:E\rightarrow E'$ defined over $F'$ where $F'$ is either $F$ or a degree $p$ extension, such that $(\ker \psi) \cap \langle P_0 \rangle=\{O\}$. Note that $j(E)\neq j(E')$ by the assumption that $E$ does not have CM. Hence $\psi(P_0)\in E'(F')$ is a point of order $p^k$. Furthermore $\psi(P_0)$ is in the kernel of the cyclic isogeny $\varphi\circ \hat\psi$ of order $p^{k+1}$ where $\varphi$ is the isogeny of $E$ such that $\ker \varphi$ is generated by $P_0$. It follows that
$$\im \rho_{E'/F', p^{k+1}}\subseteq \left\{\begin{pmatrix}
                                1+p^kt & * \\
                                0 & *                              \end{pmatrix}\right\}.$$
The fixed field of the associated isogeny character is an extension $F''/F'$ of degree 1 or $p$, and $\Gal_{F''}$ fixes a point $P''$ of order $p^{k+1}$.  As $E'(F'')$ has at least the same prime-to-$p$ torsion as $E(F)$, we conclude that $E'(F'')$ has a point of order $pN$ and hence corresponds to a point $x'$ on $X_1(pN)$ of degree dividing $p^2\deg (x)$,  as desired. Note that if $\deg(x)$ is odd, we may take $p$ to be odd by \Cref{Thm:main} and then $\deg(x')$ is also odd.

It remains to show that repeating this procedure will give in each step a $j$-invariant that hasn't been already obtained previously. Denote by $E_i$ the elliptic curve that has been obtained in the $i$th step corresponding to a sporadic point on $X_1(p^iN)$. The graph whose vertex set consists of non-CM $j$-invariants of elliptic curves connected by an isogeny of degree a power of $p$ and whose edges are between $p$-isogenous $j$-invariants is a tree, as the existence of a cycle would imply an endomorphism with nontrivial cyclic kernel and hence complex multiplication on some (and thus all) curves in the isogeny class. So all that one needs to show is that $j(E_{i-1})\neq j(E_{i+1})$. To see this, note that in our construction a multiple of the point of order $p^k$ on $E_{i}$ that we obtained will generate the kernel of the isogeny to $E_{i-1}$, but not to $E_{i+1}$, from which the claim follows.
\end{proof}

\appendix
\setcounter{secnumdepth}{0}
\section{Appendix}
We define the following subgroups of $\GL_2(\Z/18\Z)$:\\

$
H_1 \coloneqq \left\langle \big(
\begin{smallmatrix}
                                2 & 7 \\
                                 15& 16                              \end{smallmatrix}\big), \big(\begin{smallmatrix}
                                7 & 11 \\
                                 12& 11                              \end{smallmatrix}\big),  \big(\begin{smallmatrix}
                                13 & 2 \\
                                 9& 17                               \end{smallmatrix}\big),  \big(\begin{smallmatrix}
                                1 & 8 \\
                                 15& 5                               \end{smallmatrix}\big)
                                \right\rangle
$\\

$
H_2 \coloneqq \left\langle \big(\begin{smallmatrix}
                                2 & 9 \\
                                 7& 10                              \end{smallmatrix}\big), \big(\begin{smallmatrix}
                                14 & 3 \\
                                 7& 10                              \end{smallmatrix}\big),  \big(\begin{smallmatrix}
                                17 & 15 \\
                                 3& 14                               \end{smallmatrix}\big)
                                \right\rangle
$\\


$
H_3 \coloneqq \left\langle \big(\begin{smallmatrix}
                                7 & 2 \\
                                 15& 11                              \end{smallmatrix}\big), \big(\begin{smallmatrix}
                                1 & 6 \\
                                 6& 13                              \end{smallmatrix}\big),  \big(\begin{smallmatrix}
                                13 & 3 \\
                                 9& 4                               \end{smallmatrix}\big),  \big(\begin{smallmatrix}
                                2 & 3 \\
                                 15& 5                               \end{smallmatrix}\big)
                                \right\rangle
$\\


$
H_4 \coloneqq \left\langle \big(\begin{smallmatrix}
                                17 & 7 \\
                                 0& 1                              \end{smallmatrix}\big), \big(\begin{smallmatrix}
                                7 & 7 \\
                                 15& 16                              \end{smallmatrix}\big),  \big(\begin{smallmatrix}
                                1 & 9 \\
                                12 & 17                               \end{smallmatrix}\big)
                                \right\rangle
$\\


$
H_5 \coloneqq \left\langle \big(\begin{smallmatrix}
                                17 & 9 \\
                                 9& 8                              \end{smallmatrix}\big), \big(\begin{smallmatrix}
                                1 & 1 \\
                                 15& 10                              \end{smallmatrix}\big),  \big(\begin{smallmatrix}
                                5 & 5 \\
                                 12& 7                               \end{smallmatrix}\big)
                                \right\rangle
$\\


$
H_6 \coloneqq \left\langle \big(\begin{smallmatrix}
                                14 & 17 \\
                                 3& 5                              \end{smallmatrix}\big), \big(\begin{smallmatrix}
                                13 & 1 \\
                                 12& 5                              \end{smallmatrix}\big),  \big(\begin{smallmatrix}
                                11 & 8 \\
                                15& 7                               \end{smallmatrix}\big),  \big(\begin{smallmatrix}
                                17 & 1 \\
                                 12& 1                              \end{smallmatrix}\big)
                                \right\rangle
$\\


$
H_7 \coloneqq \left\langle \big(\begin{smallmatrix}
                                11 & 16 \\
                                 6& 17                              \end{smallmatrix}\big), \big(\begin{smallmatrix}
                                13 & 3 \\
                                 9& 14                              \end{smallmatrix}\big)
                                \right\rangle
$\\

$
H_8 \coloneqq \left\langle \big(\begin{smallmatrix}
                                5 & 2 \\
                                 12& 11                              \end{smallmatrix}\big), \big(\begin{smallmatrix}
                                17 &5 \\
                                 15& 16                              \end{smallmatrix}\big)
                                \right\rangle
$\\

$
H_9 \coloneqq \left\langle \big(\begin{smallmatrix}
                                11 & 1 \\
                                 3& 4                              \end{smallmatrix}\big), \big(\begin{smallmatrix}
                                7 &15 \\
                                 9& 14                              \end{smallmatrix}\big)
                                \right\rangle
$\\

$
H_{10} \coloneqq \left\langle \big(\begin{smallmatrix}
                                16 & 1 \\
                                 3& 11                              \end{smallmatrix}\big), \big(\begin{smallmatrix}
                                5 &15 \\
                                 9& 16                              \end{smallmatrix}\big)
                                \right\rangle
$

\bibliographystyle{amsplain}
\bibliography{bibliography1}

\def\cprime{$'$} \def\cprime{$'$}
\providecommand{\bysame}{\leavevmode\hbox to3em{\hrulefill}\thinspace}
\providecommand{\MR}{\relax\ifhmode\unskip\space\fi MR }
\providecommand{\MRhref}[2]{%
  \href{http://www.ams.org/mathscinet-getitem?mr=#1}{#2}
}
\providecommand{\href}[2]{#2}
\begin{thebibliography}{10}

\bibitem{abramovich}
Dan Abramovich, \emph{A linear lower bound on the gonality of modular curves},
  Internat. Math. Res. Notices (1996), no.~20, 1005--1011.

\bibitem{aoki95}
Noboru Aoki, \emph{Torsion points on abelian varieties with complex
  multiplication}, Algebraic cycles and related topics ({K}itasakado, 1994),
  World Sci. Publ., River Edge, NJ, 1995, pp.~1--22.

\bibitem{BalakrishnanEtAl}
Jennifer Balakrishnan, Netan Dogra, J.~Steffen M\"{u}ller, Jan Tuitman, and Jan
  Vonk, \emph{Explicit {C}habauty-{K}im for the split {C}artan modular curve of
  level 13}, Ann. of Math. (2) \textbf{189} (2019), no.~3, 885--944.
  \MR{3961086}

\bibitem{BP11}
Yuri Bilu and Pierre Parent, \emph{Serre's uniformity problem in the split
  {C}artan case}, Ann. of Math. (2) \textbf{173} (2011), no.~1, 569--584.
  \MR{2753610}

\bibitem{BPR13}
Yuri Bilu, Pierre Parent, and Marusia Rebolledo, \emph{Rational points on
  {$X^+_0(p^r)$}}, Ann. Inst. Fourier (Grenoble) \textbf{63} (2013), no.~3,
  957--984. \MR{3137477}

\bibitem{BC2}
Abbey Bourdon and Pete~L. Clark, \emph{Torsion points and isogenies on {CM}
  elliptic curves}, J. Lond. Math. Soc. (2) \textbf{102} (2020), no.~2,
  580--622. \MR{4171427}

\bibitem{BCS}
Abbey Bourdon, Pete~L. Clark, and James Stankewicz, \emph{Torsion points on
  {CM} elliptic curves over real number fields}, Trans. Amer. Math. Soc.
  \textbf{369} (2017), no.~12, 8457--8496.

\bibitem{BELOV}
Abbey Bourdon, \"{O}zlem Ejder, Yuan Liu, Frances Odumodu, and Bianca Viray,
  \emph{On the level of modular curves that give rise to isolated
  {$j$}-invariants}, Adv. Math. \textbf{357} (2019), 106824, 33. \MR{4016915}

\bibitem{OddDeg}
Abbey Bourdon, David Gill, Jeremy Rouse, and Lori~D. Watson, \emph{Odd degree
  isolated points on ${X}_1(n)$ with rational $j$-invariant}, preprint,
  available at arxiv.org:2006.14966.

\bibitem{BP}
Abbey Bourdon and Paul Pollack, \emph{Torsion subgroups of {CM} elliptic curves
  over odd degree number fields}, Int. Math. Res. Not. IMRN (2017), no.~16,
  4923--4961.

\bibitem{Box21}
Josha Box, \emph{Elliptic curves over totally real quartic fields not
  containing $\sqrt{5}$ are modular}, available at:
  https://arxiv.org/abs/2103.13975.

\bibitem{Box_quad21}
Josha {Box}, \emph{{Quadratic points on modular curves with infinite
  Mordell-Weil group}}, {Math. Comput.} \textbf{90} (2021), no.~327, 321--343
  (English).

\bibitem{ClarkVolcanoes}
Pete~L. Clark, \emph{{CM} elliptic curves: volcanoes, reality, and
  applications}, preprint, available at
  \url{http://alpha.math.uga.edu/~pete/Isogenies.pdf}.

\bibitem{LeastCMDeg}
Pete~L. Clark, Tyler Genao, Paul Pollack, and Frederick Saia, \emph{The least
  degree of a {CM} point on a modular curve}, to appear in J. London. Math.
  Soc, available at:
  http://alpha.math.uga.edu/$\sim$pete/least\_CM\_degree-1226.pdf.

\bibitem{CremonaNajmanQCurve}
John Cremona and Filip Najman, \emph{$\mathbb{Q}$-curves over odd degree number
  fields}, to appear in Res. Number Theory, available at arxiv.org:2004.10054.

\bibitem{DanielsGJ20}
Harris~B. Daniels and Enrique Gonz\'{a}lez-Jim\'{e}nez, \emph{On the torsion of
  rational elliptic curves over sextic fields}, Math. Comp. \textbf{89} (2020),
  no.~321, 411--435. \MR{4011550}

\bibitem{DanielsGJ}
Harris~B. Daniels and Enrique González-Jiménez, \emph{Serre’s constant of
  elliptic curves over the rationals}, Exp. Math. (2019), 1--19.

\bibitem{Daniels-Morrow}
Harris~B. Daniels and Jackon~S. Morrow, \emph{A group theoretic perspective on
  entanglements of division fields}, available at: arXiv:2008.09886.

\bibitem{DR}
P.~Deligne and M.~Rapoport, \emph{Les sch\'{e}mas de modules de courbes
  elliptiques}, Modular functions of one variable, {II} ({P}roc. {I}nternat.
  {S}ummer {S}chool, {U}niv. {A}ntwerp, {A}ntwerp, 1972), 1973, pp.~143--316.
  Lecture Notes in Math., Vol. 349. \MR{0337993}

\bibitem{Deg3Class}
Maarten Derickx, Anastassia Etropolski, Mark van Hoeij, Jackson~S. Morrow, and
  David Zureick-Brown, \emph{Sporadic cubic torsion}, to appera in Algebra
  Number Theory, available at arxiv.org:2007.13929.

\bibitem{DNS20}
Maarten {Derickx}, Filip {Najman}, and Samir {Siksek}, \emph{{Elliptic curves
  over totally real cubic fields are modular}}, {Algebra Number Theory}
  \textbf{14} (2020), no.~7, 1791--1800 (English).

\bibitem{derickxVH}
Maarten Derickx and Mark van Hoeij, \emph{Gonality of the modular curve
  {$X_1(N)$}}, J. Algebra \textbf{417} (2014), 52--71.

\bibitem{modular}
Fred Diamond and Jerry Shurman, \emph{A first course in modular forms},
  Graduate Texts in Mathematics, vol. 228, Springer-Verlag, New York, 2005.

\bibitem{elkies}
Noam~D. {Elkies}, \emph{{On elliptic \(K\)-curves}}, Modular curves and Abelian
  varieties. Based on lectures of the conference, Bellaterra, Barcelona, July
  15--18, 2002, Basel: Birkh\"auser, 2004, pp.~81--91 (English).

\bibitem{FLS15}
Nuno {Freitas}, Bao~V. {Le Hung}, and Samir {Siksek}, \emph{{Elliptic curves
  over real quadratic fields are modular}}, {Invent. Math.} \textbf{201}
  (2015), no.~1, 159--206 (English).

\bibitem{frey}
Gerhard Frey, \emph{Curves with infinitely many points of fixed degree}, Israel
  J. Math. \textbf{85} (1994), no.~1-3, 79--83.

\bibitem{GJLR}
Enrique Gonz\'{a}lez-Jim\'{e}nez and \'{A}lvaro Lozano-Robledo, \emph{On the
  minimal degree of definition of {$p$}-primary torsion subgroups of elliptic
  curves}, Math. Res. Lett. \textbf{24} (2017), no.~4, 1067--1096.

\bibitem{GJNajman}
Enrique {Gonz\'alez-Jim\'enez} and Filip {Najman}, \emph{{Growth of torsion
  groups of elliptic curves upon base change}}, {Math. Comput.} \textbf{89}
  (2020), no.~323, 1457--1485.

\bibitem{Goursat89}
Edouard Goursat, \emph{Sur les substitutions orthogonales et les divisions
  r\'eguli\`eres de l'espace}, Ann. Sci. \'Ecole Norm. Sup. (3) \textbf{6}
  (1889), 9--102 (French).

\bibitem{greenberg2014}
R.~Greenberg, K.~Rubin, A.~Silverberg, and M.~Stoll, \emph{On elliptic curves
  with an isogeny of degree 7}, Amer. J. Math. \textbf{136} (2014), no.~1,
  77--109.

\bibitem{greenberg2012}
Ralph Greenberg, \emph{The image of {G}alois representations attached to
  elliptic curves with an isogeny}, Amer. J. Math. \textbf{134} (2012), no.~5,
  1167--1196.

\bibitem{Guzvic}
Tomislav {Gu\v{z}vi\'c}, \emph{{Torsion of elliptic curves with rational
  \(j\)-invariant defined over number fields of prime degree}}, {Proc. Am.
  Math. Soc.} \textbf{149} (2021), no.~8, 3261--3275 (English).

\bibitem{RSZ21}
Andrew V.~Sutherland Jeremy~Rouse and David Zureick-Brown, \emph{{$\ell$-adic
  images of Galois for elliptic curves over $\Q$}}, preprint, available at
  arXiv:2106.11141.

\bibitem{kamienny86}
S.~Kamienny, \emph{Torsion points on elliptic curves over all quadratic
  fields}, Duke Math. J. \textbf{53} (1986), no.~1, 157--162.

\bibitem{kamienny92}
\bysame, \emph{Torsion points on elliptic curves and {$q$}-coefficients of
  modular forms}, Invent. Math. \textbf{109} (1992), no.~2, 221--229.

\bibitem{Kenku79}
M.~A. {Kenku}, \emph{{The modular curve \(X_0(39)\) and rational isogeny}},
  {Math. Proc. Camb. Philos. Soc.} \textbf{85} (1979), 21--23 (English).

\bibitem{Kenku80_2}
\bysame, \emph{{The modular curve \(X_0\)(169) and rational isogeny}}, {J.
  Lond. Math. Soc., II. Ser.} \textbf{22} (1980), 239--244 (English).

\bibitem{Kenku80}
\bysame, \emph{{The modular curves \(X_ 0(\)65) and \(X_ 0(\)91) and rational
  isogeny}}, {Math. Proc. Camb. Philos. Soc.} \textbf{87} (1980), 15--20
  (English).

\bibitem{Kenku81}
\bysame, \emph{{On the modular curves \(X_0(125)\), \(X_1(25)\) and
  \(X_1(49)\)}}, {J. Lond. Math. Soc., II. Ser.} \textbf{23} (1981), 415--427
  (English).

\bibitem{kenku}
M.~A. Kenku, \emph{On the number of {${\bf Q}$}-isomorphism classes of elliptic
  curves in each {${\bf Q}$}-isogeny class}, J. Number Theory \textbf{15}
  (1982), no.~2, 199--202. \MR{675184}

\bibitem{KM88}
M.~A. Kenku and F.~Momose, \emph{Torsion points on elliptic curves defined over
  quadratic fields}, Nagoya Math. J. \textbf{109} (1988), 125--149.

\bibitem{Lang-algebra}
Serge Lang, \emph{Algebra}, 3 ed., Graduate Texts in Mathematics, vol. 211,
  Springer-Verlag, New York, 2002.

\bibitem{lemosTrans}
Pedro Lemos, \emph{Serre's uniformity conjecture for elliptic curves with
  rational cyclic isogenies}, Trans. Amer. Math. Soc. \textbf{371} (2019),
  no.~1, 137--146. \MR{3885140}

\bibitem{lemosZ}
\bysame, \emph{Some cases of {S}erre's uniformity problem}, Math. Z.
  \textbf{292} (2019), no.~1-2, 739--762. \MR{3968924}

\bibitem{Levi1908}
Beppo Levi, \emph{Saggio per una teoria aritmetica delle forme cubiche
  ternarie}, Atti della Reale Acc. Sci. di Torino \textbf{43} (1908), 99--120,
  413--434, 672--681.

\bibitem{LRAnn}
\'{A}lvaro Lozano-Robledo, \emph{On the field of definition of {$p$}-torsion
  points on elliptic curves over the rationals}, Math. Ann. \textbf{357}
  (2013), no.~1, 279--305.

\bibitem{mazur77}
B.~Mazur, \emph{Modular curves and the {E}isenstein ideal}, Inst. Hautes
  \'Etudes Sci. Publ. Math. (1977), no.~47, 33--186 (1978).

\bibitem{mazur78}
\bysame, \emph{Rational isogenies of prime degree (with an appendix by {D}.
  {G}oldfeld)}, Invent. Math. \textbf{44} (1978), no.~2, 129--162.

\bibitem{merel}
Lo{\"{\i}}c Merel, \emph{Bornes pour la torsion des courbes elliptiques sur les
  corps de nombres}, Invent. Math. \textbf{124} (1996), no.~1-3, 437--449.

\bibitem{MilneFieldGaloisTheory}
J.S. Milne, \emph{Fields and galois theory}, available at:
  https://www.jmilne.org/math/CourseNotes/FT.pdf.

\bibitem{najman16}
F.~Najman, \emph{Torsion of rational elliptic curves over cubic fields and
  sporadic points on {$X_1(n)$}}, Math. Res. Lett. \textbf{23} (2016), no.~1,
  245--272.

\bibitem{OzmanSiksek19}
Ekin Ozman and Samir Siksek, \emph{Quadratic points on modular curves}, Math.
  Comp. \textbf{88} (2019), no.~319, 2461--2484. \MR{3957901}

\bibitem{ReverterVila}
Amadeu Reverter and N\'{u}ria Vila, \emph{Images of mod {$p$} {G}alois
  representations associated to elliptic curves}, Canad. Math. Bull.
  \textbf{44} (2001), no.~3, 313--322. \MR{1847493}

\bibitem{RouseDZB}
Jeremy Rouse and David Zureick-Brown, \emph{Elliptic curves over {$\Bbb Q$} and
  2-adic images of {G}alois}, Res. Number Theory \textbf{1} (2015), Art. 12,
  34.

\bibitem{serre72}
Jean-Pierre Serre, \emph{Propri\'et\'es galoisiennes des points d'ordre fini
  des courbes elliptiques}, Invent. Math. \textbf{15} (1972), no.~4, 259--331.

\bibitem{serre97}
\bysame, \emph{Lectures on the {M}ordell-{W}eil theorem}, third ed., Aspects of
  Mathematics, Friedr. Vieweg \& Sohn, Braunschweig, 1997, Translated from the
  French and edited by Martin Brown from notes by Michel Waldschmidt, With a
  foreword by Brown and Serre. \MR{1757192}

\bibitem{shimura}
Goro Shimura, \emph{Introduction to the arithmetic theory of automorphic
  functions}, Publications of the Mathematical Society of Japan, vol.~11,
  Princeton University Press, Princeton, NJ, 1994, Reprint of the 1971
  original, Kan\^{o} Memorial Lectures, 1. \MR{1291394}

\bibitem{silverman}
Joseph~H. Silverman, \emph{The arithmetic of elliptic curves}, second ed.,
  Graduate Texts in Mathematics, vol. 106, Springer, Dordrecht, 2009.

\bibitem{Smith2018}
Hanson Smith, \emph{Ramification in division fields and sporadic points on
  modular curves}, available at: https://arxiv.org/abs/1810.04809.

\bibitem{sutherland}
Andrew~V. Sutherland, \emph{Computing images of {G}alois representations
  attached to elliptic curves}, Forum Math. Sigma \textbf{4} (2016), e4, 79.

\bibitem{SutherlandZywina}
Andrew~V. Sutherland and David Zywina, \emph{Modular curves of prime-power
  level with infinitely many rational points}, Algebra Number Theory
  \textbf{11} (2017), no.~5, 1199--1229.

\bibitem{ZywinaImages}
David Zywina, \emph{On the possible image of the mod $\ell$ representations
  associated to elliptic curves over $\mathbb{Q}$}, available at
  arxiv.org:1508.07660.

\bibitem{zywina15}
\bysame, \emph{On the possible images of the mod $\ell$ representations
  associated to elliptic curves over $\mathbb{Q}$}, preprint, submitted
  (https://arxiv.org/abs/1508.07660).

\end{thebibliography}
\end{document}